\documentclass[11pt,a4paper,usenames,dvipsnames,reqno]{amsart}
\usepackage{amscd}
\usepackage{amsfonts,amssymb}
\usepackage{mathtools,stmaryrd}
\usepackage[margin=1in]{geometry}
\usepackage{fancybox}
\usepackage{pdflscape}
\usepackage{rotating}

\usepackage[all,arc]{xy}
\usepackage{enumerate}
\usepackage{mathrsfs}
\usepackage{color}   
\usepackage{hyperref}
\hypersetup{
    colorlinks=true, 
    linktoc=all,     
    linkcolor=blue,  
}
\usepackage{amsthm}
\usepackage{amsmath}
\usepackage{graphicx}
\usepackage{wasysym}
\usepackage{pgf,tikz}
\usetikzlibrary{positioning}
\usepackage{ marvosym }
\usepackage{tikz-cd}
\usepackage{ textcomp }
\usepackage{bbm}
\usepackage{ stmaryrd }
\usepackage{enumitem}
\usepackage{booktabs}
\usepackage{array}

\numberwithin{equation}{section}

\newcommand{\cc}{\mathbb C}

\newcommand{\zz}{\mathbb Z}

\newcommand{\qq}{\mathbb Q}
\newcommand{\rr}{\mathbb R}
\newcommand{\A}{\mathbb A}

\newcommand{\fg}{\mathfrak g}

\newcommand{\fp}{\mathfrak p}

\newcommand{\der}{\mathrm{der}}

\newcommand{\calM}{\mathcal{M}}

\DeclareMathOperator{\Hom}{Hom}

\DeclareMathOperator{\calf}{\mathcal{F}}

\makeatletter
\def\Ddots{\mathinner{\mkern1mu\raise\p@
\vbox{\kern7\p@\hbox{.}}\mkern2mu
\raise4\p@\hbox{.}\mkern2mu\raise7\p@\hbox{.}\mkern1mu}}
\makeatother

\makeatletter
\DeclareRobustCommand\bigop[1]{%
  \mathop{\vphantom{\sum}\mathpalette\bigop@{#1}}\slimits@
}
\newcommand{\bigop@}[2]{%
  \vcenter{%
    \sbox\z@{$#1\sum$}%
    \hbox{\resizebox{\ifx#1\displaystyle.9\fi\dimexpr\ht\z@+\dp\z@}{!}{$\m@th#2$}}%
  }%
}
\makeatother

\newtheorem{Thm}{Theorem}[section]
\newtheorem{Prop}[Thm]{Proposition}

\newtheorem{Lem}[Thm]{Lemma}
\newtheorem{Cor}[Thm]{Corollary}

\theoremstyle{definition}
\newtheorem{Def}[Thm]{Definition}

\theoremstyle{remark}
\newtheorem{Rem}[Thm]{Remark}
\newtheorem{Ex}[Thm]{Example}

\newcommand*{\send}{\mathcal{E}\kern -.5pt nd}

\newcommand{\quash}[1]{}

\title[Weyl algebras on Braverman-Kazhdan spaces]{Weyl algebras on Braverman-Kazhdan spaces}

\author{Chun-Hsien Hsu}
\address{Department of Mathematics\\
University of Chicago\\
Chicago, IL 60637}
\email{chunhsien@uchicago.edu}


\begin{document}
\begin{abstract}
Let $G$ be a split, simply connected, almost simple algebraic group over a field of characteristic zero, and let $P$ be a maximal parabolic subgroup of $G$. We study the ring of differential operators on $P^{\mathrm{der}}\backslash G$, showing that it shares several key structural properties with classical Weyl algebras. We also develop a corresponding theory of $D$-modules.
\end{abstract}

\maketitle
\setcounter{tocdepth}{1}
\tableofcontents

\section{Introduction}

Throughout the paper $K$ is a field of characteristic zero. For an affine $K$-variety $X$, let $\mathcal{D}(X)$ be the ring of differential operators on $X$ in the sense of Grothendieck. Explicitly, let $\mathcal{D}_0:=K[X]$ and for $i\ge 1$ let
\begin{align*}
    \mathcal{D}_i(X):=\{ D\in \mathrm{End}_K (K[X]): [D,f]:=Df-fD\in \mathcal{D}_{i-1}(X) \textrm{ for all } f\in K[X]\}.
\end{align*}
Here we have identified $f\in K[X]$ as an element in $\mathrm{End}_K (K[X])$ by function multiplication. Then $\mathcal{D}(X):=\bigcup_{i\ge 0} \mathcal{D}_i(X)$ is a subalgebra of $\mathrm{End}_K (K[X]).$ Note that $\{\mathcal{D}_i(X)\}_{i\ge 0}$ is the order filtration of $\mathcal{D}(X)$.  When $X$ is smooth, $\mathcal{D}(X)$ enjoys many desirable properties. For instance:
\begin{enumerate}
    \item The ring $\mathcal{D}(X)$ is a simple Noetherian domain.
    \item The associated graded $K$-algebra $\mathrm{gr}\,\mathcal{D}(X)$ with respect to the order filtration is naturally isomorphic to $K[T^\ast X^{\mathrm{sm}}],$ where $T^\ast X^{\mathrm{sm}}=T^\ast X$ denotes the cotangent bundle of $X$.
    \item Every finitely generated left $\mathcal{D}(X)$-module satisfies Bernstein inequality. In particular, this allows one to define the notion of holonomic $D$-module.
\end{enumerate}
\noindent These statements may no longer hold if $X$ is singular. In the present paper, we show that the above properties remain valid for a new family of singular varieties.

Let $G$ be a split, simply connected, almost simple algebraic group over $K,$ and let $P$ be a parabolic subgroup of $G$. The scheme $X_P^\circ:=P^{\mathrm{der}}\backslash G$ is strongly quasi-affine, i.e., $K[X_P^\circ]$ is a finitely generated $K$-algebra, and the natural map $X_P^\circ\to X_P:=\mathrm{Spec}\, K[X_P^\circ]$ is an open immersion. We refer to $X_P$ as a \textbf{Braverman-Kazhdan space}. The variety $X_P$ is generally singular; it is smooth precisely when it is a vector space. Nevertheless, $X_P$ is normal, and the complement $X_P-X_P^\circ$ has codimension at least $2$. Consequently, every differential operator on $X_P^\circ(K)$ extends uniquely to $X_P(K)$. 

When $K=\cc$ and $B$ is a Borel subgroup, $\mathcal{D}(X_B)$ was first studied in \cite{BBP:Gluingbasicaffine}, where it was shown that $\mathcal{D}(X_B)$ is Noetherian and that the category of $\mathcal{D}(X_B)$-modules can be described as a glued category. Properties (1)-(3) were subsequently established in \cite{Stafford, GR:basic}. A key observation in \cite{BK:basicaffine, BBP:Gluingbasicaffine} is that if $P\supset B$ is the minimal parabolic subgroup associated to a simple root of $G$, then the natural projection $X_B\to X_{P}$ is a fibration with fiber $\A^2-\{0\}$. This allows analytic difficulties to be bypassed when constructing ``Fourier automorphisms"  on $\mathcal{D}(X_B)$ by  reducing to the classical Weyl algebra $\mathcal{D}(\A^2-\{0\})=\mathcal{D}(\A^2)$. However, this fiber bundle method is special to Borel subgroups, since in this case the fibers are essentially affine spaces, whose ring of differential operators is well understood.

In the present paper, we prove that $\mathcal{D}(X_P)$ also enjoys properties (1)-(3) when $P$ is maximal and $G$ is not of type $E$ or $F$\footnote{This assumption is only used to establish the existence of $W_{X_P}$, which relies on \cite{Hsu:asymptotics}, and is not used elsewhere in the paper.}. In \cite{Hsu:asymptotics}, by  developing the necessary harmonic analysis on $X_P$ in the archimedean case, the author introduced a $\cc$-algebra $W_{X_P}$ in the spirit of \cite{BBP:Gluingbasicaffine}, and proved that $W_{X_P}$ can be defined over $\qq$ and that it embeds into $\mathcal{D}(X_P).$ The present paper serves as a continuation of this work to study $W_{X_P}$ and completes the basic theory of $D$-modules on $X_P$.

We now outline our approach. We review the definition of $W_{X_P}$ in \S \ref{sec:review}. In \S \ref{sec:basic}, we show that $W_{X_P}$ is a simple domain by identifying a distinguished collection of elements in $W_{X_P}$ that generate the identity. In the classical Weyl algebra, this is a consequence of the product rule, i.e., $[x,\partial]=1$. The existence of such elements plays a fundamental role throughout the paper. 

We next prove that every finitely generated left $W_{X_P}$-module enjoys Bernstein inequality, namely, its Gelfand–Kirillov dimension lies between $\dim X_P$ and $2\dim X_P$. In particular, this implies $W_{X_P}$ is an Ore domain, which allows us to perform a localization argument to prove $W_{X_P}=\mathcal{D}(X_P)$ in Theorem \ref{thm:Weyl=Diff}. We subsequently adapt the argument of  \cite{GR:basic} to prove $\mathrm{gr}\,\mathcal{D}(X_P)\cong K[T^\ast X^\circ_P]$ in Theorem \ref{thm:filtration} via quantization. Finally, we give an explicit description of the category of finitely generated left $\mathcal{D}(X_P)$-modules as a glued category in the sense of \cite{BBP:Gluingbasicaffine}. As an application, we prove the existence of Bernstein-Sato polynomials on $X_P,$ and establish meromorphic continuation of Igusa's local zeta functions of $X_P$ (in the archimedean setting). 

We close this introduction by relating our work to the broader Braverman-Kazhdan-Ng\^{o} program, and more generally to the Poisson summation conjecture. Roughly speaking, it is expected that on an affine normal spherical variety $X$ over $\rr$, there exists a Schwartz space $\mathcal{S}(X(\rr))$ with a suitable Fourier theory that allows one to generalize Tate's thesis, in which local $L$-functions are realized as the greatest common divisor of zeta integrals. Furthermore, one expects that $\mathcal{S}(X(\rr))$ is local, so in particular it should be stable under multiplication by the regular functions $\cc[X]$. Equivalently, this suggests that in the archimedean case
\begin{align*}
    \mathcal{S}(X(\rr))=\{ f\in L^2(X(\rr)): Df\in L^2(X(\rr)) \textrm{ for all } D\in \mathcal{D}(X)\}
\end{align*}
as proposed in \cite{BKnormalized}. The identity above is confirmed when $X=X_P$ and $P$ is maximal \cite[Theorem 7.16]{Hsu:asymptotics} as a consequence of meromorphic continuation of Igusa's local zeta functions. Interestingly, the study of $\mathcal{D}(X_P)$ also leads to a better understanding of the composition series of degenerate principal series $\mathrm{Ind}_P^G\chi$ (see \cite[\S 7.3]{Hsu:asymptotics}). These perspectives suggest a new connection between harmonic analysis on $X$ and the theory of $D$-modules on $X$, which may merit further investigation. We refer the reader to \cite{Modulation} for a more detailed proposal. 

\subsection*{Acknowledgments}
The author thanks Yiannis Sakellaridis for encouragement.

\section{Preliminaries}\label{sec:review}

Let $X$ be an affine variety over $K$. For a differential operator $D\in \mathcal{D}(X)$, its order $\mathrm{ord}(D)\in \zz_{\ge 0}$ is the smallest integer $i$ such that $D\in \mathcal{D}_{i}(X).$ Often we view $f\in K[X]$ as an element in $\mathrm{End}_K(K[X])$. To avoid confusion between the two natural ways $D$ and $f$ interact, we adopt the following notation: we write $Df\in \mathcal{D}(X)$ for the product of $D$ and $f$ in $\mathcal{D}(X),$ and write $D(f)\in K[X]$ for the result of applying $D$ to $f.$

\subsection{Weyl algebras on Braverman-Kazhdan spaces}

Let $G$ be a split, simply connected, almost simple algebraic group over $\qq$. Fix a maximal split torus $T$ and a Borel subgroup $B\supset T$ of $G$. Let $P\supseteq B$ be a maximal parabolic subgroup of $G$ and $P^{\mathrm{der}}$ be its derived subgroup. The scheme $X_P^\circ:=P^{\mathrm{der}}\backslash G$ is strongly quasi-affine, i.e., the ring $\qq[X_P^\circ]$ is finitely generated and the natural map $X_P^\circ\longrightarrow X_P:=\mathrm{Spec}\, (\qq[X_P^\circ])$ is an open immersion. The Braverman-Kazhdan space $X_P$ is a normal spherical variety.

\begin{Lem}\label{lem:XP}
    The spherical variety $X_P$ is factorial and Gorenstein.
\end{Lem}

\begin{proof}
      As $X_P$ is normal spherical, it is Cohen-Macaulay by \cite[Corollary 2.3.4]{Geometryofspherical} and \cite[\href{https://stacks.math.columbia.edu/tag/045P}{Tag 045P}]{stacks-project}. As explained in the proofs of \cite[Lemma 3.2]{Gannon:GKconjecture} and \cite[Corollary 1.3]{FuLiu}, we have an injection of divisor class groups $\mathrm{Cl}(X_P)=\mathrm{Cl}(X_P^\circ)\hookrightarrow \mathrm{Cl}(G)$ by \cite[Theorem 18.32]{Milne:AGbook}. Since $G$ is simply connected, $\mathrm{Cl}(G)=0$ by \cite[Corollary 18.24]{Milne:AGbook}. Therefore $K[X_P]$ is a UFD for any field $K$ of characteristic zero. As $X_P$ is Cohen-Macaulay and factorial, $X_P$ is Gorenstein.
\end{proof}

Let $M_P$ be the Levi subgroup of $P$ containing $T$. Then we have a natural group action
\begin{align}\label{eq:action}
\begin{split}
     X_P^\circ \times M_P^{\mathrm{ab}}\times G &\longrightarrow X_P^\circ\\
    (x,m,g)&\longmapsto m^{-1}xg.
\end{split}
\end{align}
 Let $\omega_P$ be the fundamental weight corresponding to the simple root associated to $P$. Let $V_P$ be the right $G$-representation of (anti-dominant) highest weight $-\omega_P$, and fix a highest weight vector $v_P\in V_P(\qq)$. By \cite[Lemma 3.4]{Getz:Hsu:Leslie} we have a well-defined embedding
\begin{align*}
\mathrm{Pl}:=\mathrm{Pl}_{v_P}:X_P^\circ\lhook\joinrel\xrightarrow{\quad} V_P
\end{align*}
which extends to a closed immersion $X_P\hookrightarrow V_P$, and $X_P-X_P^\circ$ is a closed point mapped to the origin under $\mathrm{Pl}$ by \cite[Theorems 1 and 2]{Popov:Vinberg}.

Let $V_P^\lor$ be the dual representation of $V_P$ and $P^{\mathrm{op}}$ be the parabolic subgroup opposite to $P$. Let $v_{P^{\mathrm{op}}}^*\in V_P^{\lor}(\qq)$ be the lowest weight vector such that $\langle v_P,v_{P^{\mathrm{op}}}^*\rangle=1 $. Similarly, we have a  $G$-equivariant closed embedding $\mathrm{Pl}_{v^\ast_{P^{\mathrm{op}}}}: X_{P^{\mathrm{op}}}\lhook\joinrel\xrightarrow{\quad}V_P^\lor$. The restriction of the canonical pairing between $V_P$ and $V_P^\lor$ yields a pairing 
\begin{align*}
   \langle \cdot ,\cdot\rangle :X_P\times X_{P^{\mathrm{op}}}\longrightarrow \mathbb{G}_a 
\end{align*}
via the embeddings $\mathrm{Pl}_{v_P},\mathrm{Pl}_{v_{P^{\mathrm{op}}}^\ast}$. This pairing is independent of the choice of $v_P$.

For $j\in\zz_{\ge 0},$ let $V_{-j\omega_P}$ be the (absolutely) irreducible representation of $G$ with highest weight $-j\omega_P$. We have  $\qq[X_P]=\bigoplus_{j\ge 0} V_{-j\omega_P}^\lor(\qq)$ and $\qq[X_{P^{\mathrm{op}}}]=\bigoplus_{j\ge 0} V_{-j\omega_P}(\qq)$. Therefore, we have a natural grading on both $\qq[X_P]$ and $\qq[X_{P^{\mathrm{op}}}]$ by degree. Let
\begin{align*}
    x_1(x)&:=\langle x, v_{P^{\mathrm{op}}}^\ast\rangle,\quad x\in X_P(\qq),\\
    x_1^{\mathrm{op}}(x^\ast)&:=\langle v_P, x^\ast\rangle,\quad x^\ast\in X_{P^{\mathrm{op}}}(\qq).
\end{align*}
Throughout we fix a choice of $\qq$-basis $\mathcal{B}$ of $\qq[X_P]$ consisting of weight vectors so that $\mathcal{B}_j:=\mathcal{B}\cap V_{-j\omega_P}^\lor(\qq)$ contains $x_1^j$. Since $V_{-j\omega_P}^\lor$ is absolutely irreducible, by Schur's lemma
\begin{align*}
    \big(V_{-j\omega_P}^\lor(\qq)\otimes_\qq V_{-j\omega_P}(\qq)\big)^{G(\qq)}=\Hom_{G(\qq)}(V_{-j\omega_P}^\lor(\qq),V_{-j\omega_P}^\lor(\qq))=\qq\mathrm{Id}.
\end{align*}
For each $b$ in $\mathcal{B}_j$, let $b^{\mathrm{op}}$ be the unique element in $V_{-j\omega_P}(\qq)$ such that
\begin{align*}
    \sum_{b\in \mathcal{B}_j} b\otimes b^{\mathrm{op}}\in \big(V_{-j\omega_P}^\lor(\qq)\otimes_\qq V_{-j\omega_P}(\qq)\big)^{G(\qq)}
\end{align*}
and $(x_1^j)^\mathrm{op}=(x_1^{\mathrm{op}})^j$. Write $\mathcal{B}_1=\{x_1,x_2,\ldots, x_{\dim V_P}\}.$

Let $\mathcal{S}(X_P(\rr))$ be the Schwartz space of $X_P(\rr)$ in \cite{Getz:Hsu:Leslie, Hsu:asymptotics}. Let $\psi(x):=e^{2\pi i x}$ be a nontrivial additive character of $\rr$. There is a unitary Fourier transform
\begin{align*}
   \mathcal{F}_{P|P^{\mathrm{op}}}=\mathcal{F}_{P|P^{\mathrm{op}},\psi}: \mathcal{S}(X_P(\rr))\longrightarrow \mathcal{S}(X_{P^{\mathrm{op}}}(\rr))
\end{align*}
such that $\mathcal{F}_{P^{\mathrm{op}}|P}\circ\mathcal{F}_{P|P^{\mathrm{op}}}=\mathrm{Id}.$  Let $\Lambda=\{(s_i,\lambda_i)\}_{i=1}^k$ be the multiset attached to $(G,P)$ as explained in \cite[\S 3.2]{Hsu:asymptotics}. For each $i,$ $(s_i,\lambda_i)\in \frac{1}{2}\zz_{\ge 0}\times \zz_{>0}.$ We fix a choice of good ordering on $\Lambda,$ i.e., $\lambda_i s_{i+1}\ge s_i\lambda_{i+1}.$ The highest data $(s_k,\lambda_k)$ is unique in the sense that $\lambda_k=1$ and $\lambda_i s_k>s_i$ for any $i<k$. Let
\begin{align*}
    J(s):=\prod_{i=1}^k \lambda_i^{\lambda_i}\prod_{\ell=0}^{\lambda_i-1}\left(s+s_k-\frac{s_i}{\lambda_i} +\frac{\ell}{\lambda_i}\right)\in s\zz[s],\quad d_P:=\sum_{i=1}^k \lambda_i.
\end{align*}

Suppose $\mathcal{S}(X_P(\rr))$ is stable under function multiplication by $\cc[X_P].$ This condition is satisfied whenever $G$ is not of type $E$ or type $F$ \cite[Proposition 7.2]{Hsu:asymptotics}. Then we have well-defined operators
\begin{align*}
    \partial_i:=(2\pi\sqrt{-1})^{d_P}\mathcal{F}_{P^{\mathrm{op}}|P}\circ x_i^{\mathrm{op}}\circ\mathcal{F}_{P|P^{\mathrm{op}}},\quad \partial_i^{\mathrm{op}}:=(2\pi\sqrt{-1})^{d_P}\mathcal{F}_{P|P^{\mathrm{op}}}\circ x_i\circ\mathcal{F}_{P^{\mathrm{op}}|P}.
\end{align*}
By \cite[Theorem 7.8]{Hsu:asymptotics} each $\partial_i$ can be realized as a differetial operator on $X_P^\circ(\rr).$ Explicitly, 
\begin{align}\label{eq:realization}
    \partial_1=c_1\frac{\partial}{\partial x_1}+\sum_{\ell=2}^{d_P} \bigg( c_\ell x_1^{\ell-1}\frac{\partial^\ell}{\partial x_1^\ell}+\sum_{\substack{\underline{j}\neq (\ell,0,\ldots,0)\\|\underline{j}|=\ell} } p_{\underline{j}} \frac{\partial^\ell}{\partial \mathbf{x}^{\underline{j}}}\bigg)
\end{align}
where $p_{\underline{j}}\in \qq[X_P]$, if nonzero, is a homogeneous polynomial of degree $|\underline{j}|-1$, and 
\begin{align*}
    c_\ell=\frac{1}{\ell!}\sum_{j=0}^{\ell} (-1)^{\ell-j}\binom{\ell}{j}J(j)\in \qq_{>0}.
\end{align*}
For a $\qq$-algebra $R$, we define the \textbf{Weyl algebra} $W_{X_P}(R)$ to be the $R$-algebra generated by $\{x_i,\partial_i\}_{i=1}^{\mathrm{dim} V_P}$. When $X_P$ is a vector space, this is the classical Weyl algebra. 

Initially, the operators $\partial_i$ are only defined as (injective) endomorphisms of $\mathcal{S}(X_P(\rr))$. In particular, in the notation of \cite{Hsu:asymptotics}, one has $W_{X_P}(\cc)=W_{X_P(\rr)}\subseteq \mathrm{End}_\cc(\mathcal{S}(X_P(\rr)))$. Thanks to \eqref{eq:realization}, $K[X_P]$ is a faithful simple left $W_{X_P}(K)$-module. Hence we may regard $W_{X_P}(K)\subseteq \mathrm{End}_K(K[X_P])$. For further details, see \cite[\S 7.2]{Hsu:asymptotics}. We record several facts that will be used repeatedly without further mention.

\begin{Lem}
    \begin{enumerate}
        \item Let $n\in \zz_{\ge 0}.$ We have $\partial_1(x_1^n)=J(n)x_1^{n-1}$ and $\partial_i(x_1^n)=0$ for all $i\neq 1.$
        \item For each nonconstant $f\in K[X_P],$ $\partial_i(f)\neq 0$ for some $i$.
        \item Let $r\ge 1$. For each $b\in \mathcal{B}_r,$ there exists (unique) $f\in V_{-r\omega_P}^\lor(\qq)$ such that $b^{\mathrm{op}}(\partial)(f)=1$ and $b'^{\mathrm{op}}(f)=0$ for all $b\neq b'\in \mathcal{B}_r$
    \end{enumerate}
\end{Lem}

\begin{proof}
    Assertion (1) follows from the computation after  Lemma 7.7 and the proof of Theorem 7.8 in \cite[\S 7.2]{Hsu:asymptotics} using highest weight theory. For (2) we may assume $f$ is homogeneous of degree $n\ge 1$. Suppose on the contrary $\partial_i(f)=0$ for all $i$. Since $V^{\lor}_{-n\omega_P}(K)$ is irreducible, by $G(K)$-action we have $\partial_i(f)=0$ for all $f\in V^{\lor}_{-n\omega_P}(K)$ and $i,$ contradicting (1). Assertion (3) follows from (2) and the fact that the natural pairing between $V_{-r\omega_P}$ and $V^\lor_{-r\omega_P}$ is perfect and $G$-equivariant.
\end{proof}

In the rest of the paper, for ease of notation we will often drop $K$ and write $W_{X_P}=W_{X_P}(K),$ $G=G(K),$ and so on, for any field $K$ of characteristic zero when there is no risk of confusion. 

\begin{Lem}\label{eq:Fourierauto}
We have a $G$-equivariant isomorphism
\begin{align*}
    \mathcal{F}_{P|P^{\mathrm{op}}}: W_{X_P}&\longrightarrow W_{X_{P^{\mathrm{op}}}}\\
    x_i&\mapsto \partial_i^{\mathrm{op}}\\
    \partial_i &\mapsto x_i^{\mathrm{op}}.
\end{align*}\qed
\end{Lem}

\section{Basic properties of Weyl algebras}\label{sec:basic}

In this section, we establish fundamental properties of the algebra $W_{X_P}$. Recall that 
\begin{align*}
    \partial_i=(2\pi\sqrt{-1})^{d_P}\mathcal{F}_{P^{\mathrm{op}}|P}\circ x_i^{\mathrm{op}}\circ\mathcal{F}_{P|P^{\mathrm{op}}},
\end{align*}
so we will often identify $\partial_i$ as coordinate functions in $\qq[X_{P^{\mathrm{op}}}].$

\begin{Lem}\label{lem:orderadd}
Let $r\in \zz_{\ge 0}$. We have $\mathrm{ord}(f(\partial))=rd_P$ for all $0\neq f\in V_{-r\omega_P}$.
\end{Lem}

\begin{proof}
      For $j\ge 0,$ let $W_j$ be the subspace in $W_{X_P}$ consisting of operators of order at most $j$. By \eqref{eq:realization} the inclusion $\qq[X_{P^{\mathrm{op}}}]\longrightarrow W_{X_P}$ induces a $G$-equivariant map 
     \begin{align*}
         V_{-r\omega_P}&\longrightarrow W_{rd_P}/W_{rd_P-1}.
     \end{align*}
     Since $\mathrm{ord} (\partial_1^r)=rd_P$ and  $V_{-r\omega_P}$ is irreducible, this map is injective, and the lemma follows.
\end{proof}

\begin{Lem}\label{lem:naivebase}
The set
\begin{align*}
    \big\{ b_1(\mathbf{x})b_2^{\mathrm{op}}(\mathbf{\partial})\in W_{X_P}: b_1,b_2\in \mathcal{B}\big\}
\end{align*}
is linearly independent over $K$.
\end{Lem}

\begin{proof}
Suppose on the contrary that there is nontrivial linear relation over $K$
\begin{align*}
    \sum_{j\in J} c_jb_{1j}(\mathbf{x})b_{2j}^{\mathrm{op}}(\mathbf{\partial})=0.
\end{align*}
Choose $J$ to be minimal.  Since $K[X_P]$ is a faithful $W_{X_P}$-module, the identity is the same as
\begin{align*}
    \sum_{j\in J} c_jb_{1j}(\mathbf{x})b_{2j}^{\mathrm{op}}(\mathbf{\partial})(f)=0
\end{align*}
for all $f\in K[X_P].$ Choose $j_0\in J$ such that $\deg b_{2j_0}^{\mathrm{op}}\le \deg b_{2j}^{\mathrm{op}}$ for all $j\in J$. Choose $f\in V_{-j_0\omega_P}^\lor$ such that $b_{2j}^{\mathrm{op}}(\mathbf{\partial})(f)=\delta_{jj_0}.$ This implies $c_{j_0}b_{1j_0}=0$ and thus $c_{j_0}=0$, a contradiction. 
\end{proof}

\begin{Prop}
\begin{enumerate}
    \item $W_{X_P}$ is a domain.
    \item The group of units in $W_{X_P}$ is $K^\times$.
    \item The center of $W_{X_P}$ is the field $K$.
\end{enumerate}
\end{Prop}

\begin{proof}
For $D\in W_{X_P}$ nonzero, choose a representative
\begin{align*}
    D=\sum_{\lambda\in I_D} p_{\lambda,D}(\mathbf{x}) \frac{\partial^{|\lambda|}}{\partial \mathbf{x}^\lambda}
\end{align*}
in the classical Weyl algebra $W_{V_P}$, where $I_D\subset \zz_{\ge 0}^{\dim V_P}$ is a (nonempty) finite subset and $p_{\lambda,D }\neq 0$ in $K[X_P]$ for all $\lambda\in I_D$.

For (1), as $W_{X_P}$ is a left primitive ring and hence a prime ring, it suffices to show $W_{X_P}$ is reduced. Suppose on the contrary that $D^2=0$ in $W_{X_P}$. Let
\begin{align*}
    \sum_{\lambda\in I} p_{\lambda,D^2}(\mathbf{x})\frac{\partial^{|\lambda|}}{\partial \mathbf{x}^\lambda}
\end{align*}
be the square of the representative of $D$ in $W_{V_P}$. Then $p_{\lambda,D^2}=0$ in $K[X_P]$. On the other hand, there is $\lambda\in I_D$ such that $|\lambda|$ is maximal, $2\lambda\in I$, and $2\lambda\neq \lambda_1+\lambda_2$ for any $\lambda_1,\lambda_2\in I_D$ unless $\lambda_1=\lambda_2=\lambda$.
Then, $p_{2\lambda, D^2}=p_{\lambda,D}^2\neq 0$ in $K[X_P]$, which is a contradiction.


For (2), suppose $D\not\in K^\times$ and there is $D'\in W_{X_P}$ such that $DD'=1$. Then $D'\not\in K^\times.$ Let
\begin{align*}
    \sum_{\lambda\in I} p_{\lambda,DD'}(\mathbf{x})\frac{\partial^{|\lambda|}}{\partial \mathbf{x}^\lambda}
\end{align*}
be the multiplication of the representatives of $D$ and $D'$ in $W_{V_P}$, so $p_{0,DD'}=1$ and  $p_{\lambda,DD'}$ are zero in $K[X_P]$ for $\lambda\neq 0$. However, there are $\lambda\in I_D,\lambda'\in I_{D'}$ both nonzero such that $\lambda+\lambda'\in I$ and   $p_{\lambda+\lambda',DD'}=p_{\lambda,D}p_{\lambda',D'}\neq 0$
is nonzero in $K[X_P]$. Therefore, such $D'$ does not exist. 

Finally for (3) suppose $D$ lies in the center of $W_{X_P}$. Let $\lambda\in I_D$ with $|\lambda|$ minimal. Assume first $|\lambda|>0$. Choose $i$ such that $\lambda_i\neq 0$. Let $e_i\in \zz_{\ge 0}^{\mathrm{dim}V_P}$ be the vector whose $i$th entry is $1$ and zero otherwise. Then $p_{\lambda-e_i,[D,x_i]}=\lambda_i p_{\lambda,D}\neq 0$
in $K[X_P]$, contradicting the assumption that $[D,x_i]=0$. Suppose $\lambda=0$ and $p_0$ is nonconstant. Then there is $i$ such that $\partial_i(p_0)\neq 0$ in $K[X_P]$, and thus $[\partial_i,p_0]\neq 0$, a contradiction.
\end{proof}

\begin{Cor}
The Weyl algebra $W_{X_P}$ is neither left nor right Artinian.
\end{Cor}
\begin{proof}
Since $W_{X_P}$ is a left primitive ring of infinite dimension over $K$ with $W_{X_P}(K)^\times=K^\times$, the assertion is a consequence of the Wedderburn-Artin theorem. 
\end{proof}

\subsection{Simplicity of Weyl algebras}

Let 
\begin{align*}
    E:=\sum_{i=1}^{\dim V_P} x_i\frac{\partial}{\partial x_i}.
\end{align*}
be the usual Euler operator on $V_P$. Note that $E$ is a differential operator on $X_P^\circ,$ which corresponds to the differential of the $M^{\mathrm{ab}}$-action on $X_P^\circ$ \eqref{eq:action}.
\begin{Lem}\label{Lem:unitgen}
Let $r\in \zz_{>0}$. We have
\begin{align*}
    \sum_{b\in \mathcal{B}_r} b(\mathbf{x})b^{\mathrm{op}}(\mathbf{\partial})&=Z_r(E),\\
     \sum_{b\in \mathcal{B}_r} b^{\mathrm{op}}(\mathbf{\partial}) b(\mathbf{x})&=Z_r(-E-2s_k-2),
\end{align*}
where
\begin{align*}
    Z_r(t):=\prod_{n=0}^{r-1}J(t-n).
\end{align*}
\end{Lem}

\begin{proof}
We may assume $K=\qq$. Observe that for each $j\ge 0,$ $E$ acts on $V_{-j\omega_P}^\lor$ by the scalar multiplication by $j$. Similarly, the operator
\begin{align*}
    L:=\sum_{b\in \mathcal{B}_r} b(\mathbf{x})b^{\mathrm{op}}(\partial)
\end{align*}
is invariant under the $G$-action, and hence by Schur's lemma $L$ acts on $V_{-j\omega_P}^\lor$ by some scalar $c_j$ for each $j\ge 0$. Explicitly,
\begin{align*}
c_jx_1^j=L(x_1^j)=x_1^r\partial_1^r(x_1^j)=\prod_{n=0}^{r-1} J(j-n)x_1^j=Z_r(j)x_1^j=Z_r(E)(x_1^j).
\end{align*}
Since both $Z_r(E)$ and $L$ are differential operators of order $rd_P$ and their actions agree on $\qq[X_P],$ we have $L=Z_r(E)$.

Now let 
\begin{align*}
R:= \sum_{b\in \mathcal{B}_r} b^{\mathrm{op}}(\mathbf{\partial}) b(\mathbf{x}).
\end{align*}
Consider its action on $\mathcal{S}(X_{P}(\rr))$. For $f\in\mathcal{S}(X_{P^{\mathrm{op}}}(\rr)),$ applying  the computation above to $P^{\mathrm{op}},$ we have
\begin{align*}
    R\circ \mathcal{F}_{P^{\mathrm{op}}|P}(f)=\mathcal{F}_{P^{\mathrm{op}}|P}\circ L(f)=\mathcal{F}_{P^{\mathrm{op}}|P}\circ Z_r(E)(f).
\end{align*}
Note that $E$ can be identified as an element in $\mathfrak{m}^{\mathrm{ab}}.$ For $m\in M^{\mathrm{ab}}(\rr),$ let $m.(\cdot)$ denote the natural action of $M^{\mathrm{ab}}(\rr)$ on $\mathcal{S}(X_{P}(\rr))$ and $\mathcal{S}(X_{P^{\mathrm{op}}}(\rr))$ induced by \eqref{eq:action}. Then  
\begin{align*}
    \mathcal{F}_{P^{\mathrm{op}}|P}(m.f)=\delta_P(m)m.\mathcal{F}_{P^{\mathrm{op}}|P}(f),
\end{align*}
where $\delta_P$ is the modular character of $P$. Taking the differential of the above identity, we have by \cite[Proposition 6.2]{Getz:Hsu:Leslie}
\begin{align*}
    \mathcal{F}_{P^{\mathrm{op}}|P}\circ E(f)=(-2s_k-2-E)\circ \mathcal{F}_{P^{\mathrm{op}}|P}(f).
\end{align*}
This implies 
\begin{align*}
     R\circ \mathcal{F}_{P^{\mathrm{op}}|P}(f)=Z_r(-E-2s_k-2)\circ \mathcal{F}_{P|P^{\mathrm{op}}}(f)
\end{align*}
for any $f\in\mathcal{S}(X_{P^{\mathrm{op}}}(\rr)).$ The lemma follows.
\end{proof}


\begin{Lem}\label{lem:coprime}
For any $r_1,r_2\in \zz_{>0}$. The polynomials $Z_{r_1}(t), Z_{r_2}(-t-2s_k-2)$ are coprime.
\end{Lem}

\begin{proof}
Roots of $J(t-n)$ are
\begin{align*}
    n-s_k+\frac{s_i}{\lambda_i}-\frac{\ell}{\lambda_i}>n-s_k-1.
\end{align*}
While roots of $J(-t-2s_k-2-n)$ are
\begin{align*}
    -2s_k-2-n+s_k-\frac{s_i}{\lambda_i}+\frac{\ell}{\lambda_i}<-s_k-1-n.
\end{align*}
Therefore, roots of $Z_{r_1}(t)$ and $Z_{r_2}(-t-2s_k-2)$ are disjoint.
\end{proof}

\begin{Prop}\label{prop:E}
We have $E\in W_{X_P}(\qq)$.
\end{Prop}

\begin{proof}
By Lemma \ref{Lem:unitgen} the assertion follows if  there exists $N\in \zz_{>0}$ such that the $\qq$-subalgebra of $\qq[t]$ generated by $\{Z_{j}(t), Z_j(-t-2s_k-2): 1\le j\le N\}$  is $\qq[t]$. By base change, we can prove the statement over $\cc$. It suffices to show there is $N$ such that the map 
\begin{align*}
    \Phi_N: \cc&\to \cc^{2N}\\
    t&\mapsto (Z_j(t), Z_j(-t-2s_k-2))_{1\le j\le N}
\end{align*}
is injective and $\frac{d}{dt}\Phi_N$ is nonvanishing. Since $Z_1(t)$ and $Z_1(-t-2s_k-2)$ have rational roots in $(-s_k-1,0]$ and $[-2s_k-2,-s_k-1)$ respectively, by Rolle's theorem $\frac{d}{dt}\Phi_1$ is nonvanishing and so is $\frac{d}{dt}\Phi_N$ for any $N$.

Because the fraction field of $\cc[Z_1(t), Z_1(-t-2s_k-2)]$ is $\cc(t)$, the map $\Phi_1:\cc\to C_1:=\mathrm{im}(\Phi_1)$ is the normalization of the affine curve $C_1$. If $C_1$ is smooth then we are done. Suppose $C_1$ is singular and let  $S$ be its singular locus. Given $(x,y)\in S$, let $t_1\neq t_2$ such that $\Phi_1(t_1)=\Phi_1(t_2)=(x,y)$. We claim there exists $N$ such that either $Z_{N}(t_1)\neq Z_N(t_2)$ or $Z_{N}(-t_1-2s_k-2)\neq Z_{N}(-t_2-2s_k-2)$. Then since $|S|<\infty$ and $\Phi_1$ is a finite morphism, $\Phi_N$ is injective for $N$ sufficiently large, and the proposition follows.

To see the claim, since $t_1\neq t_2$ and $J(t)$ is a polynomial, we can choose $N$ large so that $J(t_1-N)\neq J(t_2-N).$ If $Z_{N}(t_1)\neq Z_{N}(t_2),$ then we are done. Suppose $Z_{N}(t_1)= Z_{N}(t_2)$. If they are both nonzero, then
\begin{align*}
    Z_{N+1}(t_1)=Z_{N}(t_1)J(t_1-N)\neq Z_{N}(t_2)J(t_2-N)=Z_{N+1}(t_2).
\end{align*}
If they are both zero, then both $Z_N(-t_1-2s_k-2)$ and $Z_{N}(-t_2-2s_k-2)$ are nonzero by Lemma \ref{lem:coprime}. Using the same argument, we can choose $N'\ge N$ large so that $Z_{N'}(-t_1-2s_k-2)\neq Z_{N'}(-t_2-2s_k-2)$. This proves the claim.
\end{proof}

\begin{Cor}\label{cor:critical}
    Let $M\subseteq M'$ be left $W_{X_P}$-modules. Let $m\in M'$. Suppose there is $r\in \zz_{\ge 1}$ such that 
    \begin{enumerate}
        \item either $b(\mathbf{x})m, b^{\mathrm{op}}(\partial)m\in M$  for all $b\in \mathcal{B}_r$, or
        \item both $Z_r(E)m$ and $Z_r(-E-2s_k-2)m$ are contained in $M$.
    \end{enumerate}
    Then $m\in M$. In particular, if in (1) or (2) all terms are zero, then $m=0$.
\end{Cor}

\begin{proof}
    By Lemma \ref{Lem:unitgen}, both assumptions imply $Z_r(E)m, Z_r(-E-2s_k-2)m\in M$. Then $m\in M$ by Lemma \ref{lem:coprime} and Proposition \ref{prop:E}. The last assertion is a special case where $M=0$.
\end{proof}

\begin{Cor}\label{cor:E}
The isomorphism  $\mathcal{F}_{P|P^{\mathrm{op}}}: W_{X_P}\longrightarrow W_{X_{P^{\mathrm{op}}}}$
induces a $K$-automorphism of the subring $K[E]$ given by $E+s_k+1\mapsto -(E+s_k+1)$.\qed
\end{Cor}

\begin{Thm}\label{thm:simple}
The Weyl algebra $W_{X_P}$ is a simple ring.
\end{Thm}

\begin{proof}
Let  $I$ be a nonzero two-sided ideal of $W_{X_P}$. For $r\ge 0,$ let $\mathrm{span}(\mathcal{B}_r^{\mathrm{op}}(\partial))$ be the $K$-span of the set $\{b^{\mathrm{op}}(\partial):b\in \mathcal{B}_r\}.$ We claim  $\mathrm{span}(\mathcal{B}_r^{\mathrm{op}}(\partial))\cap I$ is nonzero for some $r\ge 0$. To see this, by applying the isomorphism $\mathcal{F}_{P|P^{\mathrm{op}}},$ it amounts to show $K[X_P]\cap I$ is a nonzero homogeneous ideal. Since $I$ is a two-sided ideal of $W_{X_P}$ and $E\in W_{X_P}$ by Proposition \ref{prop:E}, $K[X_P]\cap I$ is a homogeneous ideal. Let $D$ be a nonzero element in $I$. If $\mathrm{ord}(D)=0$, then we are done. Otherwise, there is $i$ such that $0\neq [D,x_i]\in I$ and $\mathrm{ord}([D,x_i])<\mathrm{ord}(D)$. Successively, one obtains a nonzero element in $K[X_P]\cap I$. 

Let $j=r(d_P-1).$ Consider the $G$-equivariant multilinear map
 \begin{align*}
      V_{-r\omega_P}\times (V_P^\lor)^{j+r}&\to V^\lor_{-j\omega_P}\\
      (b^{\mathrm{op}}(\partial),b_1(\mathbf{x}),\ldots, b_{j+r}(\mathbf{x}))&\mapsto [[\ldots[[b^{\mathrm{op}}(\partial),b_1(\mathbf{x})],b_2(\mathbf{x})],\ldots], b_{j+r}(\mathbf{x})].
 \end{align*}
Since $K[X_P]$ is a commutative subring of $W_{X_P},$ by the Jacobi identity the map descends to a morphism $\Phi\in\Hom_G(V_{-r\omega_P}\otimes \mathrm{Sym}^{j+r} V_P^\lor, V_{-j\omega_P}^\lor)$ such that $\Phi(\partial_1^r \otimes x_1^{j+r})=c x_1^{j}$ for some $c\neq 0$.
As $\mathrm{span}(\mathcal{B}_r^{\mathrm{op}}(\partial))\cap I$ is nonzero, Lemma \ref{lem:fullentry} below implies $\mathrm{span}(\mathcal{B}_{r(d_P-1)})\subseteq I$. Therefore, by Lemma \ref{Lem:unitgen} $Z_{r(d_P-1)}(E)$ and $Z_{r(d_P-1)}(-E-2s_k-2)$ are contained in $I$. Corollary \ref{cor:critical} implies $1\in I$.
\end{proof}

\begin{Lem}\label{lem:fullentry}
Let $r\in \zz_{\ge 0}$. For $j\ge 0$, let
\begin{align*}
  \Phi\in \Hom_G(V_{-r\omega_P}\otimes \mathrm{Sym}^{j+r} V_P^\lor, V_{-j\omega_P}^\lor)
\end{align*}
be a map such that $\Phi(\partial_1^r \otimes x_1^{j+r})=c x_1^{j}$ for some $c\neq 0$. For any nonzero $v\in V_{-r\omega_P}$, we have $$\Phi\Big(v\otimes \mathrm{Sym}^{j+r} V_P^\lor\Big)=V_{-j\omega_P}^\lor.$$
\end{Lem}

\begin{proof}
Let $v\in V_{-r\omega_P}$ be nonzero. Consider the vector subspace  
\begin{align*}
     U_v:=\Phi\Big(v\otimes \mathrm{Sym}^{j+r} V_P^\lor\Big)\subseteq V_{-j\omega_P}^\lor.
\end{align*}
Assume first $v=\partial_1^r$. Then $U_v$ is stable under $P$ and by assumption contains the lowest weight vector $x_1^j$, so $U_v=V_{-j\omega_P}^\lor$ by irreducibility. It follows that $U_v=V_{-j\omega_P}^\lor$ for all $v\in G(K)\partial_1^r$. Observe that the condition for $v\neq 0$ satisfying $U_v\subsetneq V_{-r\omega_P}^\lor $ defines a closed $G$-stable subvariety $Y$ of $\mathbb{P}V_{-r\omega_P}$. If $Y$ is nonempty, then $Y$ contains the unique closed orbit $P\backslash G$ in $\mathbb{P}V_{-r\omega_P},$ which is a contradiction.
\end{proof}

\section{$D$-modules}

Let $R$ be a finitely generated $K$-algebra. A family of $K$-vector spaces $\mathcal{F}=(F_m)_{m\ge 0}$ is a \textbf{filtration} of $R$ if
it is an exhaustive ascending chain of $K$-vector subspaces of $R$ and satisfies
\begin{align*}
    F_m\cdot F_n\subseteq F_{m+n} \textrm{ for all } m,n\ge 0.
\end{align*}
In this case we say $R$ is a \textbf{filtered algebra}, and we will write $(R,\mathcal{F})$ if we want to specify the filtration $\calf$. We adopt the convention that $F_m=0$ if $m<0$. Let
\begin{align*}
    \mathrm{gr}\, R=\mathrm{gr}^{\calf}(R):=\bigoplus_{m\ge 0} F_m/F_{m-1}
\end{align*}
be the associated graded algebra. Given a set of generators $Z$ of $R$, we say a filtration $\calf$ of $R$ is \textbf{standard} with respect to $Z$ if $F_0=K$ and
\begin{align*}
    F_m:=F_1^m:=\mathrm{span}_K\{ D_1\cdots D_r: D_i\in Z, r\le m\}.
\end{align*}

Let $M$ be a left $R$-module. By a filtration $\Gamma=(\Gamma_i)_{i\ge 0}$ on $M$ with respect to $(R,\mathcal{F})$, we mean an exhaustive ascending chain of $K$-vector subspaces of $M$ such that
\begin{align*}
    F_m\cdot \Gamma_i\subseteq \Gamma_{i+m} \textrm{ for all }i,m\ge 0.
\end{align*}
A filtration $\Gamma$ is \textbf{good} if $\dim\Gamma_i:=\dim_K\Gamma_i<\infty$ for all $i$, and $\Gamma$ is \textbf{standard} if $\Gamma_i=F_i\cdot \Gamma_0.$ Let
\begin{align*}
    \mathrm{gr}^{\Gamma}(M):=\bigoplus_{i\ge 0} \Gamma_i/\Gamma_{i-1}
\end{align*}
be the associated graded module of $ \mathrm{gr}^{\calf}(R)$. Here $\Gamma_{-1}=0$.

We recall the definition of Gelfand-Kirillov dimension in this setting and record its standard properties (see for instance \cite[Chapter 8]{noncomm}).

\begin{Def}
Let $Z\subseteq R$ be a finite set of generators and $\mathcal{F}$ be a standard filtration of $R$ with respect to $Z$. Let $M$ be a left $R$-module with a good standard filtration $\Gamma$ (with respect to $\mathcal{F}$) with $\dim_K\Gamma_0<\infty$. We define the \textbf{Gelfand-Kirillov dimension} (or simply GK-dimension) of $M$, to be
\begin{align*}
    \mathrm{GK}(M):=\limsup_{k\to \infty} \log_k \dim \Gamma_k.
\end{align*}
\end{Def}

\begin{Prop}\label{prop:easy}
Let $M$ be a finitely generated left $R$-module with a good standard filtration $\Gamma$. We have the following.
\begin{enumerate}
    \item $\mathrm{GK}(M)$ does not depend on $Z$ and $\Gamma_0$.
    \item $\mathrm{GK}(M)=\mathrm{GK}(\mathrm{gr}^\Gamma(M))$.
    \item Given an exact sequence of finitely generated $R$-modules
    $$0\to M_1\longrightarrow  M\longrightarrow  M_2\longrightarrow  0,$$
    we have $\mathrm{GK}(M)\ge \sup\{\mathrm{GK}(M_1),\mathrm{GK}(M_2)\}$ and the equality holds if the exact sequence splits.
    \item $\mathrm{GK}(M)\le \mathrm{GK}(R)$.
\end{enumerate}\qed
\end{Prop}

\subsection{Bernstein inequality}

We define the \textbf{Bernstein filtration} $\mathrm{Ber}=(B_k)_{k\ge 0}$ on $W_{X_P}$ to be the standard filtration with respect to the set
\begin{align*}
    \{x_1,\ldots, x_{\dim V_{P}},\partial_1,\ldots,\partial_{\dim_{V_P}}\}.
\end{align*}
The associated graded algebra $\mathrm{gr}^{\mathrm{Ber}}(W_{X_P})$ is commutative if and only if $d_P=1,$ which is exactly when $X_P$ is smooth, i.e., $X_P$ is a vector space. We note that the isomorphism $\mathcal{F}_{P|P^{\mathrm{op}}}$ respects the Bernstein filtration.

\begin{Lem}\label{lem:prebernstein}
Let $M$ be a nonzero finitely generated left $W_{X_P}$-module. Suppose $\Gamma$ is a filtration on $M$ with respect to the Bernstein filtration such that $\Gamma_0\neq 0$. Then there is an absolute constant $c\in \zz_{\ge 1}$ such that the natural map
\begin{align*}
    B_k\longrightarrow \mathrm{Hom}_K\big(\Gamma_{kc},\Gamma_{kc+k}\big)
\end{align*}
is injective for all $k\ge 0.$
\end{Lem}

\begin{proof}
By \cite[Lemma 9.4.1]{C:Dmod} we can take $c=1$ if $d_P=1$. Assume $d_P\ge 2$. Let $e$ be the smallest integer such that $E\in B_{e}$, and let $c=d_P^2+(d_P-1)^2+2e.$ The assertion is clear for $k=0$ by hypothesis. Suppose on the contrary that there is $k>0$ such that $D\Gamma_{kc}=0$ for some nonzero $D\in B_k$. Write $D=\sum D_\ell$ such that $[D_\ell,E]=\ell D_\ell.$ Then $D_\ell\neq 0$ only if $-k\le \ell\le k$. Choose $\ell$ such that $D_\ell\neq 0$. Since $D_\ell$ is a finite linear combination of elements of the form
\begin{align*}
   [[\ldots[D,\underbrace{E],\ldots], E]}_{r \textrm{ times}}
\end{align*}
for $0\le r\le 2k,$ we have 
\begin{align*}
    D_\ell\Gamma_{kc-2ke}\subseteq K[E]D\Gamma_{kc}=0.
\end{align*}
By \eqref{eq:realization}, there is a nonnegative integer $m\le kd_P$ and  $b_1,\ldots,b_m\in \mathcal{B}_1$ such that 
\[p:=[[\ldots[D_\ell,b_1(\mathbf{x})],\ldots], b_{m}(\mathbf{x})]\in K[X_P]\]
is nonzero homogeneous of degree at most $k(d_P-1).$ Then
\begin{align*}
    p\Gamma_{kc-kd_P-2ke}=[[\ldots[D_\ell,b_1(\mathbf{x})],\ldots], b_{m}(\mathbf{x})]\Gamma_{kc-kd_P-2ke}\subseteq K[X_P]D_\ell\Gamma_{kc-2ke} =0.
\end{align*}
Arguing as in the second part of the proof of Theorem \ref{thm:simple} and using the fact that $\partial_i \Gamma_n\subseteq\Gamma_{n+1},$ we have 
\begin{align*}
    b^{\mathrm{op}}(\partial)\Gamma_{kc-kd_P-2ke-k(d_P-1)d_P}=0
\end{align*}
for all $b\in \mathcal{B}_r,$ where $r=(\mathrm{deg}\, p)(d_P-1)\le k(d_P-1)^2$. Then by Lemma \ref{Lem:unitgen} 
\begin{align*}
    Z_r(E)\Gamma_{kc-k(d_P^2+(d_P-1)^2+2e)}= Z_r(-E-2s_k-2)\Gamma_{kc-k(d_P^2+(d_P-1)^2+2e)}=0.
\end{align*}
Corollary \ref{cor:critical} implies $\Gamma_{kc-k(d_P^2+(d_P-1)^2+2e)}=\Gamma_0=0,$ a contradiction.
\end{proof}

Let $H_{X_P}$ be the Hilbert polynomial of the graded ring $K[X_P]$, i.e., $H_{X_P}\in\qq[t]$ is the unique degree $\dim X_P$ polynomial such that for $k$ large
\begin{align*}
    H_{X_P}(k)=\sum_{j=0}^k\dim V_{-j\omega_P}^\lor.
\end{align*}

\begin{Cor}\label{cor:polygrowth}
There is an absolute constant $c>0$ such that for $k$ large
\begin{align*}
    \dim B_k\le H_{X_P}\big(kc\big)H_{X_P}\big(kc+k\big).
\end{align*}
\end{Cor}

\begin{proof}
Apply Lemma \ref{lem:prebernstein} to the module $K[X_P]$ and observe that the degree filtration on $K[X_P]$ is the standard filtration on $K[X_P]$ with respect to the Bernstein filtration with $\Gamma_0=K$. 
\end{proof}

\begin{Thm}[Bernstein inequality]\label{thm:Bernstein}
Let $M$ be a nonzero finitely generated left $W_{X_P}$-module. We have $$2\,\mathrm{dim} X_P=\mathrm{GK}(W_{X_P})\le2 \, \mathrm{GK}(M) .$$ 
 Furthermore, there are positive constants $c,k_0$ (independent of $M$) such that for all $k\ge k_0$
\begin{align*}
    \dim \Gamma_{k}\ge ck^{\dim X_P}.
\end{align*}
\end{Thm}

\begin{proof}
 Let $\Gamma=(\Gamma_k)$ be a good standard filtration on $M$ with respect to the Bernstein filtration. By Lemma \ref{lem:prebernstein}, there is an absolute constant $c>0$ such that $\dim B_k\le (\dim \Gamma_{kc}) (\dim \Gamma_{kc+k})$. Therefore, by definition $\mathrm{GK}(W_{X_P})\le 2 \, \mathrm{GK}(M).$ 
 
 By Corollary \ref{cor:polygrowth}, $\mathrm{GK}(W_{X_P})\le 2\,\mathrm{dim} X_P$. Observe that $B_k$ contains
\begin{align*}
    \{ b_1(\mathbf{x})b_2^{\mathrm{op}}(\mathbf{\partial})\in W_{X_P}: b_1,b_2\in \mathcal{B}, \deg b_1+\deg b_2\le k\}.
\end{align*}
Since the Hilbert polynomial of $K[X_P\times X_{P^\mathrm{op}}]$ is of degree $2\, \mathrm{dim} X_P$, we have $2\,\mathrm{dim} X_P\le \mathrm{GK}(W_{X_P})$ by Lemma \ref{lem:naivebase}. The second statement follows easily.
\end{proof}

\begin{Cor}\label{cor:Ore}
The Weyl algebra $W_{X_P}$ is a left and right Ore domain.  
\end{Cor}

\begin{proof}
    Since $W_{X_P}$ is a domain with finite GK-dimension, this follows from \cite[Theorem 8.1.21]{noncomm}.
\end{proof}

\subsection{Weyl algebra as the ring of differential operators} 
We identify elements in $\mathfrak{g}$ as (first order) linear differential operators on $X_P^\circ$ induced by the differential of the action \eqref{eq:action}.

\begin{Prop}\label{prop:liecontained}
    We have $\mathfrak{g}\subseteq W_{X_P}$.
\end{Prop}

\begin{proof}
 By Corollary \ref{cor:Ore}, we can localize $W_{X_P}$ by the multiplicative set $S:=\{x_1^n:n\in \zz_{\ge 0}\}$. By the universal property of localization \cite[Corollary 2.1.4]{noncomm} we have
 \begin{align*}
     (W_{X_P})_{x_1}:=S^{-1}W_{X_P}=W_{X_P}S^{-1}.
 \end{align*}
By \eqref{eq:realization} there is a constant $c\in \qq_{>0}$ such that
\begin{align*}
    cx_1^{d_P-1}\frac{\partial}{\partial x_1}=[[\ldots[[\partial_1,\underbrace{x_1],x_1],\ldots], x_1]}_{d_P-1 \textrm{ times}}\in W_{X_P}.
\end{align*}
Thus $ \tfrac{\partial}{\partial x_1}\in  (W_{X_P})_{x_1}.$ Since $x_1$ is a lowest weight vector and $\tfrac{\partial}{\partial x_1}$ is a highest weight vector, by the action of $P^{\mathrm{op}}$ we have that for any $i,$ $\tfrac{\partial}{\partial x_i}\in (W_{X_P})_{x_1}.$ Hence $\mathfrak{g}\subseteq  (W_{X_P})_{x_1}.$ 

Note that $\mathfrak{\fg}$ is the adjoint representation of $G$. Let $D\in \fg$ be a highest weight vector. The discussion above implies there is $r\ge 0$ such that $x_1^rD, Dx_1^r\in W_{X_P}$. As $V_{-r\omega_P}^\lor$ is irreducible, by the action of $B,$ $b(\mathbf{x})D,  Db(\mathbf{x}) \in W_{X_P}$ for all $b\in \mathcal{B}_r.$   Hence $DZ_r(E),Z_r(-E-2s_k-2)D\in W_{X_P}$ by Lemma \ref{Lem:unitgen}. Since $D$ commutes with $E$, we have $DZ_r(E)=Z_r(E)D$. It follows from Corollary \ref{cor:critical} that $D\in W_{X_P}$. As $D$ generates $\fg$ as a $G$-representation, the proposition follows.
\end{proof}

Since the definition of the ring of differential operators $\mathcal{D}$ is compatible with localization, we have a sheaf of differential operators $\mathcal{D}=\mathcal{D}_X$ on any quasi-projective $K$-variety $X$. It is a quasi-coherent sheaf of $\mathcal{O}_X$-modules.  Let $\mathcal{T}$ be the tangent sheaf of $X.$  Suppose $X$ is affine. One has $$\mathcal{D}_1(K[X])=K[X]+\mathcal{T}(X).$$ Let $\Delta(X)\subseteq \mathrm{End}_K(K[X])$ be the $K$-subalgebra generated by $\mathcal{D}_1(X),$ so $\Delta(X)\subseteq \mathcal{D}(X)$. In general the containment is strict, but when $X$ is smooth, $\Delta(X)=\mathcal{D}(X)$. For more discussions, we refer the reader to \cite[Chapter 15]{noncomm}.

\begin{Thm}\label{thm:Weyl=Diff}
    We have $\mathcal{D}(X_P)= W_{X_P}.$
\end{Thm}
\begin{proof}
     Since $X_P-X_P^\circ$ has codimension at least $2$ and $X_P$ is normal, we have $\mathcal{D}(X_P)=\mathcal{D}(X_P^\circ)$. The containment $W_{X_P}\subseteq \mathcal{D}(X_P^\circ)$ follows from \eqref{eq:realization}. Thus we only need to show $\mathcal{D}(X_P)\subseteq W_{X_P}.$ 
     
     Let $D\in \mathcal{D}(X_P).$ We may assume $[E,D]=\ell D$ for some integer $\ell$. We claim there is $r\in \zz_{\ge 0}$ such that  $b(\mathbf{x})D, Db(\mathbf{x})\in W_{X_P}$ for all $b\in \mathcal{B}_r$. Note that $\mathcal{D}(X_P)$ is equipped with a natural rational $G$-action from \eqref{eq:action}. Assume first $\cc D$ is fixed by $B$. Let $U:=V(x_1)^c:=\{x_1\neq 0\}\subset X_P.$ It is an affine open subvariety of $X_P^\circ$. Since $\mathcal{T}(U)=\mathcal{T}(X_P^\circ)_{x_1}$ is generated by $\mathfrak{g}$ as a $K[U]$-module, Proposition \ref{prop:liecontained} implies
\begin{align*}
    \Delta(U)\subseteq  (W_{X_P})_{x_1}\subseteq \mathcal{D}(U).
\end{align*}
As $U$ is smooth, $\Delta(U)=\mathcal{D}(U)$ and the inclusions above are equalities. In particular, $x_1^rD$ and $Dx_1^r$ are in $W_{X_P}$ for some $r$ large. Since $\cc D$ is fixed by $B,$ by the $B$-action we have $b(\mathbf{x})D,  Db(\mathbf{x}) \in W_{X_P}$  for all $b\in \mathcal{B}_r.$ For general $D$, consider the finite dimensional $G$-submodule $V\subseteq \mathcal{D}(X_P)$ generated by $D$. As $V$ is semisimple, the general case follows from the theory of highest weight.

By the previous discussion and Lemma \ref{Lem:unitgen}, $Z_r(-E-2s_k-2)D, DZ_r(E)\in  W_{X_P}.$
Note that $DZ_r(E)=Z_r(E-\ell)D.$ If $\ell\ge 0,$ then $Z_{r}(t-\ell)$ and $Z_{r}(-t-2s_k-2)$ are coprime. Thus $D\in W_{X_P}$. Assume $\ell<0$. For any $b\in \mathcal{B}_{-\ell}$, $[E,b(\mathbf{x})D]=0$ and thus  $b(\mathbf{x}) D\in W_{X_P}.$ Therefore, Lemma \ref{Lem:unitgen} implies $$Z_{-\ell}(-E-2s_k-2)D\in W_{X_P}.$$ For the same reason, for any $r\in \zz_{\ge 0}$ and $b\in \mathcal{B}_r$
\begin{align*}
    Z_{r-\ell}(-E-2s_k-2)b^{\mathrm{op}}(\partial)D\in W_{X_P}.
\end{align*}
Since for any $h\in K[t]$ and $b\in \mathcal{B}_r$, $b^{\mathrm{op}}(\partial)h(E)=h(E+r)b^{\mathrm{op}}(\partial),$
we have
\begin{align*}
    W_{X_P}&\ni b^{\mathrm{op}}(\partial)Z_{-\ell}(-E-2s_k-2)D=Z_{-\ell}(-E-2s_k-2+r)b^{\mathrm{op}}(\partial)D.
\end{align*}
For $r$ sufficiently large, $Z_{-\ell}(t+r)$ and $Z_{r-\ell}(t)$ are coprime, so $b^{\mathrm{op}}(\partial)D\in W_{X_P}$. Consequently, there is $r$ large such that $b(\mathbf{x})D$ and $b^{\mathrm{op}}(\partial)D$ for all $b\in \mathcal{B}_r$. Corollary \ref{cor:critical} implies $D\in W_{X_P}$.
\end{proof}

\begin{Lem}\label{lem:structure}
    We have natural isomorphisms of $W_{X_P}$-modules
    \begin{align*}
        W_{X_P}/W_{X_P}\mathfrak{g}&\cong K[X_P] \oplus K[X_{P^{\mathrm{op}}}]
    \end{align*}
    and
    \begin{align*}
    W_{X_P}/W_{X_P}(\mathfrak{g}+\mathfrak{m}_P^{\mathrm{ab}})\cong K[X_P].
    \end{align*}
\end{Lem}

\begin{proof}
Consider the homomorphism of $W_{X_P}$-modules 
 \begin{align*}
        \varphi:W_{X_P}/W_{X_P}\mathfrak{g}&\longrightarrow K[X_P] \oplus K[X_{P^{\mathrm{op}}}]\\
    D&\mapsto (D(1),\mathcal{F}_{P|P^{\mathrm{op}}}(D)(1)).
    \end{align*}
Let $r\ge 1$ and $b\in \mathcal{B}_r$. Then $\varphi(b(\mathbf{x}))=(b,0)$ and $\varphi(b^{\mathrm{op}}(\mathbf{\partial}))=(0,b^{\mathrm{op}}).$ Also $\varphi(1)=(1,1)$ and $\varphi(E)=(0,-2(s_k+1))$ by Corollary \ref{cor:E}. Thus $\varphi$ is surjective. 

Let $D\in W_{X_P}$ such that $D(1)=0$. Note that for a smooth affine variety $X$, $K[X]\cong \mathcal{D}(X)/\mathcal{D}(X)\mathcal{T}(X).$ Thus by localization and the argument in the first part of the proof of Theorem \ref{thm:Weyl=Diff}, there exists $r\ge 0$ such that $b(\mathbf{x})D\in W_{X_P}\mathfrak{g}$ for all $b\in \mathcal{B}_r$. Therefore, if $\varphi(D)=0,$ there is some $r$ such that  $b(\mathbf{x})D, b^{\mathrm{op}}(\partial)D\in W_{X_P}\mathfrak{g}$ for all $b\in \mathcal{B}_r$. Corollary \ref{cor:critical} implies that $D\in W_{X_P}\mathfrak{g},$ so $\varphi$ is an isomorphism. As a consequence, the following map is also an isomorphism
\begin{align*}
W_{X_P}/W_{X_P}(\mathfrak{g}+\mathfrak{m}_P^{\mathrm{ab}})&\longrightarrow K[X_P] \\
    D&\mapsto D(1).
\end{align*}
\end{proof}

\begin{Prop}
We have an isomorphism of $W_{X_P}$-modules 
\begin{align*}
    H^{i}(X_P^\circ,\mathcal{O}_{X_P^\circ})\cong\begin{cases}
        K[X_P] & \textrm{if } i=0,\\
        K[X_{P^{\mathrm{op}}}] & \textrm{if } i=\dim X_P-1,\\
        0 & \textrm{if } i\neq 0,\dim X_P-1.
    \end{cases} 
\end{align*}
Therefore,
\begin{align*}
    H^\ast(X_P^\circ,\mathcal{O}_{X_P^\circ})\cong W_{X_P}/W_{X_P}\mathfrak{g}.
\end{align*}
\end{Prop}

\begin{proof}
    Since $\mathcal{O}_{X_P^\circ}$ is a left $\mathcal{D}_{X_P^\circ}$-module,  $H^i(X_P^\circ,\mathcal{O}_{X_P^\circ})$ is a $W_{X_P}$-module for all $i$. We compute $H^i(X_P^\circ,\mathcal{O}_{X_P^\circ})$ as $K[X_P]$-modules with compatible $\mathfrak{g}$-action.
    
    Clearly $H^0(X_P^{\circ},\mathcal{O}_{X_P^\circ})=K[X_P]$. Let $\mathfrak{p}\subset K[X_P]$ be the maximal ideal generated by $\mathcal{B}_1.$ Let $i>0$. By \cite[4.6.2]{Weibel}, we have an isomorphism of $K[X_P]$-modules
    \begin{align*}
        H^i(X_P^\circ,\mathcal{O}_{X_P^\circ})\cong H^{i+1}_{\mathfrak{p}}(K[X_P]),
    \end{align*}
    where the latter is the local cohomology. By \cite[Proposition 2.14(1)]{Localcohomology}, 
    $$H^{i+1}_{\mathfrak{p}}(K[X_P])\subseteq (H^{i+1}_{\mathfrak{p}}(K[X_P]))_\mathfrak{p}\cong H^{i+1}_{\mathfrak{p}K[X_P]_\mathfrak{p}}(K[X_P]_\mathfrak{p}).$$
   Since $K[X_P]_\mathfrak{p}$ is Cohen-Macaulay,  $H^{i+1}_{\mathfrak{p}K[X_P]_\mathfrak{p}}(K[X_P]_\mathfrak{p})=0$ unless $i=\dim X_P-1$ by \cite[Corollary 4.6.9]{Weibel}. Thus $H^i(X_P^\circ,\mathcal{O}_{X_P^\circ})=0$ unless $i=\dim X_P-1$.

    Now we compute $U:=H^{\dim X_P}_\mathfrak{p}(K[X_P]).$ Consider the space $K[x_1^{-1},\ldots, x_{\dim V_P}^{-1}]$ with the $K[V_P]$-module structure given by 
    \begin{align*}
        \left(\prod_{i} x_i^{a_i}\right) \left(\prod_{i} x_i^{-b_i}\right):=\begin{cases}
            \prod_{i} x_i^{-(b_i-a_i)} & \textrm{if } b_i\ge a_i \textrm{ for all }i,\\
            0& \textrm{otherwise.}
        \end{cases}
    \end{align*}
    It has a natural $\mathfrak{g}$-action given by derivation. Let $I$ be the kernel of the map $K[V_P]\to K[X_P]$. The variety $X_P$ is Gorenstein by Lemma \ref{lem:XP}, so by the Grothendieck local duality \cite[\href{https://stacks.math.columbia.edu/tag/0AAK}{Tag 0AAK}]{stacks-project} and the computation in \cite{Matsumura:injectivehull}
    \begin{align*}
        U\cong \{ f\in K[x_1^{-1},\ldots, x_{\dim V_P}^{-1}]: I\cdot f=0\}
    \end{align*}
    as $K[X_P]$-modules with compatible $\mathfrak{g}$-actions. As a $\mathfrak{g}$-module, $U$ is isomorphic to $K[X_{P^{\mathrm{op}}}].$ 
    Let $e_0$ be a nonzero element in $H^{\dim X_P-1}(X_P^\circ,\mathcal{O}_{X_P^\circ})$ that generates the trivial $\mathfrak{g}$-representation. It follows that the kernel of the homomorphism
    \begin{align*}
        W_{X_P}&\to H^{\dim X_P-1}(X_P^\circ,\mathcal{O}_{X_P^\circ})\\
            D&\mapsto D(e_0)
    \end{align*}
    contains $W_{X_P}\mathfrak{g}$ and $x_1,\ldots, x_{\dim V_P}.$ Therefore, by Lemma \ref{lem:structure} $H^{\dim X_P-1}(X_P^\circ,\mathcal{O}_{X_P^\circ})$ contains a $W_{X_P}$-submodule isomorphic to $K[X_{P^{\mathrm{op}}}]$. Hence $H^{\dim X_P-1}(X_P^\circ,\mathcal{O}_{X_P^\circ})\cong K[X_{P^{\mathrm{op}}}]$.
\end{proof}

\subsection{Associated graded algebra of $W_{X_P}$}

Now we study the associated graded algebra of $W_{X_P}$ with respect to the order filtration via quantization. Our discussion is largely motivated by \cite{GR:basic}. 

For $\lambda\in \zz,$ we let $$W_{X_P,\lambda}:=\{D\in W_{X_P}: [E,D]=\lambda D\}.$$ 
Let $(F_m^{\mathrm{ord}})_{m\ge 0}$ be the order filtration on $W_{X_P},$ and put $$F_m^{\mathrm{ord}}W_{X_P,\lambda}:=F_m^{\mathrm{ord}}W_{X_P}\cap W_{X_P,\lambda}.$$
For $m\in \zz,$ let 
\begin{align*}
    \widehat{F}_mW_{X_P}:=\bigoplus_{\lambda\in \zz}  \widehat{F}_mW_{X_P,\lambda}:=\bigoplus_{\lambda\in \zz}\{ D\in W_{X_P,\lambda}:  \operatorname{ord}\bigl(\mathcal{F}_{P|P^{\mathrm{op}}}(D)\bigr)
\le m+\lambda d_P\big\}.
\end{align*}

\begin{Lem}\label{lem:=0}
    We have $\widehat{F}_mW_{X_P}=0$ for $m<0$ and $\widehat{F}_0W_{X_P}= K[X_P].$
\end{Lem}

\begin{proof}
    Let $m\le 0$.  Suppose there is $0\neq D\in\widehat{F}_m W_{X_P}$. We may assume $D\in W_{X_P,\lambda}$ and let \(n=\operatorname{ord}(D)\). If $n=0$, then $D\in V_{-\lambda\omega_P}^\lor,$ so by Lemma \ref{lem:orderadd} $\mathrm{ord}(\mathcal{F}_{P|P^{\mathrm{op}}}(D))=\lambda d_P\ge m+\lambda d_P.$ When $m<0,$ the inequality is strict, which contradicts the definition of $\widehat{F}_m W_{X_P}.$ If \(n>0\), there is $b_1,\ldots,b_n\in \mathcal{B}_1$ such that 
$$
p=[[\cdots[D,b_1],b_2],\dots,b_n]\in K[X_P]
$$
is a nonzero homogeneous function of degree $\lambda+n$. Then by Lemma \ref{lem:orderadd}
\begin{align*}
(\lambda+n)d_P=\mathrm{ord} (\mathcal{F}_{P|P^{\mathrm{op}}}(p))\le
m+\lambda d_P+nd_P-n,
\end{align*}
which is impossible as $m\le 0$.
\end{proof}

\begin{Lem}\label{lem:map}
For $m\in \zz_{\ge 0}$ and $\lambda\in \zz,$  $\mathcal F_{P\mid P^{\mathrm{op}}}
\left(F_m^{\mathrm{ord}}W_{X_P,\lambda}\right)
\subseteq
F_{m+\lambda d_P}^{\mathrm{ord}}W_{X_{P^{\mathrm{op}}},-\lambda}.$
\end{Lem}

\begin{proof}
We show that for any $m\in \zz$
\begin{align}\label{eq:easyF2}
    \widehat{F}_mW_{X_P}\subseteq F_m^{\mathrm{ord}}W_{X_P}.
\end{align}
When $m\le 0,$ this follows from Lemma \ref{lem:=0}. Suppose $m>0.$ Let $D\in \widehat{F}_mW_{X_P,\lambda}$. By the definition of order and induction, it suffices to show $[D,f]\in \widehat{F}_{m-1}W_{X_P}$ for all $f\in K[X_P].$ We may assume $0\neq f\in V^{\lor}_{-j\omega_P}.$ Then
\begin{align*}
    \operatorname{ord}(\mathcal{F}_{P|P^{\mathrm{op}}}([D,f]))&=\operatorname{ord} ([\mathcal{F}_{P|P^{\mathrm{op}}}(D),\mathcal{F}_{P|P^{\mathrm{op}}}(f)])\\
    &\le (m+\lambda d_P)+jd_P-1 =(m-1)+(\lambda+j)d_P.
\end{align*}
Since $[D,f]\in W_{X_P,\lambda+j},$ $[D,f]\in \widehat F_{m-1}W_{X_P}$. This proves \eqref{eq:easyF2}.

Let $D\in F_m^{\mathrm{ord}}W_{X_P,\lambda}$ so that $ A:=\mathcal{F}_{P|P^{\mathrm{op}}}(D)\in W_{X_{P^{\mathrm{op}}},-\lambda}$. As
\begin{align*}
\mathrm{ord}(\mathcal{F}_{P^{\mathrm{op}}|P}(A))=\mathrm{ord}_P(D)\le m=
(m+\lambda d_P)+(-\lambda)d_P,
\end{align*}
we have $A\in \widehat{F}_{m+\lambda d_P}W_{X_{P^{\mathrm{op}}}}.$ Applying \eqref{eq:easyF2} for $P^{\mathrm{op}},$ we conclude that $A\in F_{m+\lambda d_P}^{\mathrm{ord}}W_{X_{P^{\mathrm{op}}},-\lambda}.$ This proves the lemma.
\end{proof}


For a filtered $K$-algebra $(R,\mathcal{F}),$ its Rees algebra is a subalgebra of $R[\hbar]$ given by
\begin{align*}
    R_{\hbar}:=\bigoplus_{i\in \zz_{\ge 0}} R_{\hbar}^{i}:=\bigoplus_{i\in \zz_{\ge 0}} \hbar^{i}F_i.
\end{align*}
It is a graded $K[\hbar]$-algebra. One has a natural isomorphism of graded $K$-algebras
\begin{align*}
    R_{\hbar}/\hbar R_{\hbar}\cong \mathrm{gr}^{\mathcal{F}}R.
\end{align*}
Apply this construction to $W_{X_P}=\mathcal{D}(X_P)$ with respect to the order filtration.  Define for $D\in W_{X_P,\lambda}$ of order at most $m$
\begin{align*}
        \mathcal{F}_{P|P^{\mathrm{op}},\hbar}(\hbar^mD):=\hbar^{m+d_P\lambda} \mathcal{F}_{P|P^{\mathrm{op}}}(D).
\end{align*}
By Lemma \ref{lem:map}, this defines an homorphism of graded $K[\hbar]$-algebra 
\begin{align*}
        \mathcal{F}_{P|P^{\mathrm{op}},\hbar}: W_{X_P,\hbar}\longrightarrow W_{X_{P^{\mathrm{op}}},\hbar},
\end{align*}
and $\mathcal{F}_{P^{\mathrm{op}}|P,\hbar}\circ \mathcal{F}_{P|P^{\mathrm{op}},\hbar}=\mathrm{Id}.$ This induces an isomorphism
\begin{align}\label{eq:Fhdescends}
    \mathcal{F}_{P|P^{\mathrm{op}}}: \mathrm{gr}\, W_{X_P}\longrightarrow \mathrm{gr}\, W_{X_{P^{\mathrm{op}}}}
\end{align}

 Apply the construction of the Rees algebra to the sheaf $\mathcal{D}=\mathcal{D}_{X_P^\circ}$ with respect to the order filtration to obtain a sheaf of graded $K[\hbar]$-algebras $\mathcal{D}_{\hbar}.$ Let
\begin{align*}
    \mathcal{A}=\mathcal{A}_{X_P^\circ}:=\mathcal{D}_{\hbar}/\hbar\mathcal{D}_{\hbar}.
\end{align*}
Note that $\mathcal{A}$ is canonically isomorphic to $p_\ast \mathcal{O}_{T^\ast X_P^\circ}$, where $p:T^\ast X_P^\circ\to X_P^\circ$ is the natural projection of the cotangent bundle. Let $\mathcal{D}_{\hbar}(X_P^\circ):=\Gamma(X_P^\circ,\mathcal{D}_{\hbar})=\mathcal{D}(X_P^\circ)_{\hbar}$. Then $\mathcal{D}_{\hbar}(X_P^\circ)=W_{X_P,\hbar}$ by Theorem \ref{thm:Weyl=Diff}, and we have a natural inclusion of algebras
\begin{align*}
    \varphi:\mathrm{gr}\,W_{X_P}=\mathrm{gr}\, \mathcal{D}_{\hbar}(X_P^\circ)\cong \mathcal{D}_{\hbar}(X_P^\circ)/\hbar\mathcal{D}_{\hbar}(X_P^\circ)\hookrightarrow\Gamma(X_P^\circ,\mathcal{A})=K[T^\ast X_P^\circ].
\end{align*}

\begin{Thm}\label{thm:filtration}
Endow $W_{X_P}$ with the order filtration. Then we have a natural isomorphism of graded $K$-algebras
\begin{align*}
    \mathrm{gr}\,W_{X_P}=\mathrm{gr}\, \mathcal{D}(X_P)\xrightarrow{\sim} K[T^\ast X_P^\circ].
\end{align*}
In particular, the map $\mathcal{F}_{P|P^{\mathrm{op}}}:W_{X_P}\longrightarrow W_{X_{P^{\mathrm{op}}}}$ induces an isomorphism of $K$-algebras
    \begin{align*}
    \mathcal{F}_{P|P^{\mathrm{op}}}:K[T^\ast X_P^\circ]\longrightarrow K[T^\ast X_{P^\mathrm{op}}^\circ]
\end{align*}
such that $\mathcal{F}_{P^{\mathrm{op}}|P}\circ \mathcal{F}_{P|P^{\mathrm{op}}}=\mathrm{Id}.$\qed
\end{Thm}

\begin{proof}
We may assume $K=\cc$. For ease of notation, we let $X:=X_P^\circ$ and write $\otimes=\otimes_\cc.$ Note that we have a natural isomorphism $T^\ast X\cong (\mathfrak{p}^{\mathrm{der}}\backslash \mathfrak{g})^\vee\times ^{P^{\mathrm{der}}} G,$ so $T^\ast X$ is a $P^{\mathrm{ab}}=M^{\mathrm{ab}}$-torsor of $\widetilde{X}:= (\mathfrak{p}^{\mathrm{der}}\backslash \mathfrak{g})^\vee\times ^P G.$ Let $\mathcal{P}:=P\backslash G$ and $\mathcal{O}_{\mathcal{P}}(\lambda)$ be the line bundle on $\mathcal{P}$ associated to $\lambda\omega_P$ for $\lambda\in \zz.$ For any variety $Y$ over $\mathcal{P},$ let $\mathcal{O}_{Y}(\lambda)$ be the pullback to $Y$ of $\mathcal{O}_{\mathcal{P}}(\lambda).$ By the natural $G$-equivariant map $\widetilde{X}\longrightarrow \mathcal{P},$ we have 
\begin{align*}
    \cc[T^\ast X]\cong \bigoplus_{\lambda\in \zz} \Gamma(\widetilde{X},\mathcal{O}_{\widetilde{X}}(\lambda))= \bigoplus_{\lambda\in \zz} \mathrm{Ind}_P^G(\cc[(\fp^\der\backslash \fg)^\lor]\otimes \cc_{\lambda}),
\end{align*}
where $\cc_\lambda$ is the $1$-dimensional representation $\lambda\omega_P$ of $P$ and $\mathrm{Ind}_P^G$ is the induction functor (c.f. \cite[I.3]{Jantzen:rep}).

Note that $\cc[E,\hbar]$ is a subalgebra of $\mathcal{D}_\hbar(X).$ Decompose
\begin{align*}
    \mathcal{D}_{\hbar}(X)=\bigoplus_{\lambda \in \zz} \mathcal{D}_{\hbar}(X)_\lambda
\end{align*}
as eigenspaces of the operator $[E,\cdot].$ Then 
\begin{align*}
     \mathrm{gr}\, \mathcal{D}(X)\cong \bigoplus_{\lambda\in \zz}  \mathcal{D}_{\hbar}(X)_\lambda/\hbar \mathcal{D}_{\hbar}(X)_\lambda.
\end{align*}
Since $\varphi$ is $M^{\mathrm{ab}}\times G$-equivariant, by the $M^{\mathrm{ab}}$-action, $\varphi$ is surjective if and only if for all $\lambda\in \zz$ 
\begin{align*}
    \mathcal{D}_{\hbar}(X)_\lambda/\hbar \mathcal{D}_{\hbar}(X)_\lambda\longrightarrow\mathrm{Ind}_P^G(\cc[(\fp^\der\backslash \fg)^\lor]\otimes\cc_{\lambda})
\end{align*}
is surjective. 

For $\lambda\in \zz$ and a left $\cc[E,\hbar]$-module $M$, let $^{(\lambda)}M$ be the left $\cc[E,\hbar]$-module whose underlying space is the $\cc[\hbar]$-module $M$ with a twisted $E$-action: 
\begin{align*}
    E\cdot^{(\lambda)} m:=(E+(\lambda+s_k+1)\hbar)\cdot m, \quad\forall m\in M. 
\end{align*} 
Then 
\begin{align*}
    \mathcal{D}_{\hbar}(X)_\lambda/\hbar \mathcal{D}_{\hbar}(X)_\lambda=\bigoplus_{\lambda\in \zz}{}^{(\lambda)}\mathcal{D}_{\hbar}(X)_\lambda/\hbar {}^{(\lambda)}\mathcal{D}_{\hbar}(X)_\lambda.
\end{align*}
Therefore, to prove the theorem, we only need to show $\varphi$ restricts to an isomorphism
\begin{align*}
    {}^{(\lambda)}\mathcal{D}_{\hbar}(X)_\lambda/\hbar{}^{(\lambda)}\mathcal{D}_{\hbar}(X)_\lambda\cong  \mathrm{Ind}_P^G(\cc[(\fp^\der\backslash \fg)^\lor]\otimes\cc_{\lambda}).
\end{align*}

The module $\cc[G]$ admits a regular $G\times G$-action. Let $U(\mathfrak{g})$ be the universal enveloping algebra of $\mathfrak{g}$ and $U_{\hbar}(\fg)$ be its Rees algebra with respect to the PBW filtration. Differentiating the left regular action induces an algebra homomorphism $U_{\hbar}(\fg)\longrightarrow \mathcal{D}_\hbar(\fg),$ which induces a left $\cc[G]$-module isomorphism 
\begin{align*}
    \cc[G]\otimes U_{\hbar}(\mathfrak{g})\cong D_{\hbar}(G).
\end{align*}
Any element $D\in D_{\hbar}(G)\cong \cc[G]\otimes U_{\hbar}(\mathfrak{g})$ induces a morphism $\cc[X_P][\hbar]\longrightarrow \cc[G][\hbar]$. The morphism is trivial iff $D\in \cc[G]\otimes \left(  U_{\hbar}(\mathfrak{g})\cdot\fp^\der\right).$ Therefore,
\begin{align}\label{eq:DtoU}
    \mathcal{D}_{\hbar}(X)\cong \left(\cc[G]\otimes \left( U_{\hbar}(\mathfrak{g}) /  U_{\hbar}(\mathfrak{g})\cdot\fp^\der\right) \right)^{P^{\mathrm{der}}}.
\end{align}
Under this isomorphism, the $G$-action on $D_\hbar(X)$ translates to the right regular $G$-action on $\cc[G].$ The $P$-action on $D_{\hbar}(X),$ where $P^{\mathrm{der}}$ acts trivially, equips $U_{\hbar}(\mathfrak{g}) /U_{\hbar}(\mathfrak{g})\cdot\fp^\der$ with a $P$-module structure. Let us describe this $P$-module structure in an alternative way.

Let $\lambda\in \zz$. The algebra isomorphism $\cc[E,\hbar]\cong \fp^{\mathrm{der}}\cdot U_{\hbar}(\fp)\backslash U_{\hbar}(\fp)$ gives  $^{(\lambda)}\cc[E,\hbar]$ a bi-$U_{\hbar}(\mathfrak{p})$-module structure for $\lambda\in \zz.$ Consider $$M(\lambda):=U_\hbar(\fg)\otimes_{{U_h(\fp)}}{}^{(\lambda)}\cc[E,\hbar].$$  The action of $U_\hbar(\fg)$ on $v_\lambda:=1\otimes 1\in M(\lambda)$ induces an isomorphism of left $U_\hbar(\fg)$-modules
\begin{align*}
    M(\lambda)\cong U_{\hbar}(\fg)/U_{\hbar}(\fg)\cdot \fp^{\mathrm{der}}.
\end{align*}
Under this isomorphism, $E$ acts on the left on $U_{\hbar}(\fg)/U_{\hbar}(\fg)\cdot \fp^{\mathrm{der}}$ by $E+(\lambda+s_k+1)\hbar.$ Then the $P$-module action on $M(\lambda)$ is the exponential of the action 
\begin{align*}
    (E-(s_k+1)\hbar)\cdot m-m\cdot E, \quad \forall m\in M(\lambda). 
\end{align*}
It follows that \eqref{eq:DtoU} induces an isomorphism of graded $\cc[E,\hbar]$-modules
\begin{align*}
     {}^{(\lambda)}\mathcal{D}_{\hbar}(X)_\lambda\cong \mathrm{Ind}_P^G(M(\lambda)).
\end{align*}

In conclusion to prove the assertion we only need to show the natural short exact sequence
\begin{align*}
    0\longrightarrow \hbar M(\lambda)\longrightarrow M(\lambda)\longrightarrow M(\lambda)/\hbar M(\lambda)\cong \cc[(\fp^\der\backslash \fg)^\lor]\otimes\cc_{\lambda}\longrightarrow 0
\end{align*}
induces a short exact sequence
\begin{align*}
    0\longrightarrow \hbar \mathrm{Ind}_P^G (M(\lambda))\longrightarrow \mathrm{Ind}_P^G(M(\lambda))\longrightarrow \mathrm{Ind}_P^G(\cc[(\fp^\der\backslash \fg)^\lor]\otimes\cc_{\lambda})\longrightarrow 0.
\end{align*}

Assume $\lambda\ge 0$. It suffices to show $R^1\mathrm{Ind}_P^G(M(\lambda))=0.$ Let $(U_{\hbar}^{\le i}(\fg))_{i\ge 0}$ be the PBW filtration of $U_{\hbar}(\fg)$. For $i\ge 0,$ let $M_i:=U_{\hbar}^{\le i}(\fg)\cdot v_\lambda\cdot \cc[E,\hbar]$. Each $M_i$ is $P$-stable and $\cc[E,\hbar]$-stable, and $(M_i)_{i\ge 0}$ is a filtration of $M(\lambda)$ whose associated graded module is isomorphic to $\cc[(\fp^\der\backslash \fg)^\lor][\hbar]\otimes\cc_{\lambda}.$ In particular, we have a filtration
\begin{align*}
    R^1\mathrm{Ind}_P^G(M(\lambda))= \bigcup _{i\ge 0} R^1\mathrm{Ind}_P^G M_i.
\end{align*}
By \cite[Theorem 2.2]{Broer} for $\lambda\ge 0$
\begin{align*}
    H^i(T^\ast \mathcal{P},\mathcal{O}_{T^\ast \mathcal{P}}(\lambda))=0
\end{align*}
for all $i>0$. Since $T^\ast \mathcal{P}\cong  (\mathfrak{p}\backslash \mathfrak{g})^\lor \times^PG,$ we have $H^i(\widetilde{X} ,\mathcal{O}_{\widetilde{X}}(\lambda))=0$
for all $i>0$. Equivalently
\begin{align*}
    R^i\mathrm{Ind}_P^G(\cc[(\fp^\der\backslash \fg)^\vee]\otimes \cc_{\lambda})=0.
\end{align*}
This implies $R^i\mathrm{Ind}_P^G(M(\lambda))=0$ as
its associated graded module is  $R^i\mathrm{Ind}_P^G(\cc[(\fp^\der\backslash \fg)^\lor]\otimes\cc_{\lambda})\otimes \cc[\hbar]=0.$ 

Previous discussions imply that $\varphi$ restricting to nonnegative weight spaces of $M^{\mathrm{ab}}$
\begin{align*}
    \varphi|_{\lambda\ge 0}: (\mathrm{gr} \, W_{X_P})_{\lambda \ge 0}=\bigoplus_{\lambda\ge 0} (\mathrm{gr} \,W_{X_P})_{\lambda}\longrightarrow\cc[T^\ast X_P^\circ]_{\lambda\ge 0}
\end{align*}
is an isomorphism. For any $i,$ consider localization by $x_i$. We have $$\cc[T^\ast X_P^\circ]_{x_i}=(\cc[T^\ast X_P^\circ]_{\lambda\ge 0})_{x_i}\cong ((\mathrm{gr} \, W_{X_P})_{\lambda \ge 0})_{x_i}=(\mathrm{gr}\, W_{X_P})_{x_i}.$$ Thus \eqref{eq:Fhdescends} induces an injective homomorphism of $\cc$-algebras
\begin{align*}
    \mathcal{F}_{P|P^{\mathrm{op}}}:\cc[T^\ast X_P^\circ]_{x_i}\cong (\mathrm{gr} \, W_{X_P})_{x_i}\longrightarrow (\mathrm{gr} \, W_{X_{P^{\mathrm{op}}}})_{\partial_i^{\mathrm{op}}}\hookrightarrow \cc[T^\ast X_{P^{\mathrm{op}}}^\circ]_{\varphi(\partial_i^{\mathrm{op}})}.
\end{align*}
Consider the commutative diagram
\begin{center}
    \begin{tikzcd}
        0 \ar[r] & \cc[T^\ast X_P^\circ] \ar[r]\ar[d,dotted] & \bigoplus_{i=1}^{\dim V_P}\cc[T^\ast X_P^\circ]_{x_i}\ar[r]\ar[d,"\mathcal{F}_{P|P^{\mathrm{op}}}"] &\bigoplus_{i,j} \cc[T^\ast X_P^\circ]_{x_ix_j}\ar[d,"\mathcal{F}_{P|P^{\mathrm{op}}}"]\\
        0 \ar[r] & \cc[T^\ast X_{P^{\mathrm{op}}}^\circ] \ar[r] & \bigoplus_{i=1}^{\dim V_P}\cc[T^\ast X_{P^\mathrm{op}}^\circ]_{\varphi(\partial_i^{\mathrm{op}})}\ar[r] &\bigoplus_{i,j}\cc[T^\ast X_{P^{\mathrm{op}}}^\circ]_{\varphi(\partial_i^{\mathrm{op}})\varphi(\partial_j^{\mathrm{op}})}.
    \end{tikzcd}
\end{center}
Since $\cc[T^\ast X_P^\circ]$ and  $\cc[T^\ast X_{P^{\mathrm{op}}}^\circ]$ are Noetherian \cite{FuLiu} and integrally closed, by the algebraic Hartogs' lemma the rows are exact. In particular, the dotted map is defined
\begin{align*}
    \mathcal{F}_{P|P^{\mathrm{op}}}:\cc[T^\ast X_P^\circ]\longrightarrow\cc[T^\ast X_{P^\mathrm{op}}^\circ].
\end{align*}
Furthermore, $\mathcal{F}_{P^{\mathrm{op}}|P}\circ \mathcal{F}_{P|P^{\mathrm{op}}}=\mathrm{Id}.$ Thus $ \mathcal{F}_{P|P^{\mathrm{op}}}$ is an isomorphism. By Corollary \ref{cor:E}, for any $\lambda\le 0$ we have a commutative diagram
\begin{center}
    \begin{tikzcd}
        (\mathrm{gr} \, W_{X_P})_{\lambda} \ar[rr,"\varphi"]\ar[d,"\mathcal{F}_{P|P^{\mathrm{op}}}"] & &\cc[T^\ast X_P^\circ]_{\lambda} \ar[d,"\mathcal{F}_{P|P^{\mathrm{op}}}"]\\ (\mathrm{gr} \, W_{X_{P^{\mathrm{op}}}})_{-\lambda} \ar[rr,"\varphi"]& &\cc[T^\ast X_{P^{\mathrm{op}}}^\circ]_{-\lambda}.
    \end{tikzcd}
\end{center}
As the bottom horizontal map is an isomorphism, the top map is also an isomorphism. This completes the proof.
\end{proof}

\subsection{Category of D-modules.}
Let $\mathcal{M}(\mathcal{D}_{X_{P}^\circ})$ be the category of quasi-coherent sheaves of left $D$-modules on $X_{P}^\circ$ and $\mathcal{M}(W_{X_P})$ be the category left $W_{X_P}$-modules. For each $\mathcal{E}\in \calM(\mathcal{D}_{X_P^\circ}),$  $\Gamma(X_P^\circ,\mathcal{E})$ is a left $\mathcal{D}(X_P^\circ)=\mathcal{D}(X_P)$-module, so by Theorem \ref{thm:Weyl=Diff} we have a global section functor
\begin{align*}
    \Gamma: \calM(\mathcal{D}_{X_P^\circ}) &\longrightarrow \calM(W_{X_P})\\
     \mathcal{E} & \mapsto \Gamma(X_P^\circ, \mathcal{E}).
\end{align*}
On the other hand, we have a localization functor 
\begin{align*}
    \mathcal{L}: \calM(W_{X_P}) \longrightarrow \calM(\mathcal{D}_{X_P^\circ})      
\end{align*}
by associating to each $M\in \calM(W_{X_P})$ a sheaf $\mathcal{L}(M)$ such that for an affine open subscheme $U\subset X_P^\circ$, 
\begin{align*}
\mathcal{L}(M)(U)=\mathcal{D}(U)\otimes_{ \mathcal{D}(X_P)}M=K[U]\otimes_{K[X_P]} M.
\end{align*}
Therefore, $\mathcal{L}$ is exact, $\mathcal{L}\circ \Gamma=\mathrm{Id},$ and $\mathcal{L}$ is the left adjoint of $\Gamma,$ i.e.,
\begin{align*}
    \mathrm{Hom}_{\mathcal{M}(\mathcal{D}_{X_P^\circ})}(\mathcal{L}(M),\mathcal{E})=\mathrm{Hom}_{\mathcal{M}(W_{X_P})}(M,\Gamma(\mathcal{E})).
\end{align*}
We use the same notations for functors defined for $X_{P^{\mathrm{op}}}^\circ.$

For a left $W_{X_P}$-module $M$, we let $M^w$ be the left $W_{X_{P^{\mathrm{op}}}}$-module whose underlying space is $M$, and its module structure is given by
\begin{align*}
    x_i^{\mathrm{op}}\cdot m:=\partial_i\cdot m,\quad \partial_i^{\mathrm{op}}\cdot m:=x_i\cdot m
\end{align*}
for all $m\in M^w$ and $i$. Therefore, we have a natural equivalence functor
\begin{align*}
    w: \mathcal{D}(W_{X_P})&\longrightarrow \mathcal{D}(W_{X_{P^\mathrm{op}}})\\
    M&\longmapsto M^w.
\end{align*}
We let $w^{-1}$ be its inverse. Set
\begin{align*}
    \mathcal{L}_w&:=\mathcal{L}\circ w:\mathcal{M}(W_{X_P})\longrightarrow \mathcal{M}(\mathcal{D}_{X_{P^{\mathrm{op}}}^\circ}),\\
    \Gamma_w&:=w^{-1}\circ \Gamma:\mathcal{M}(\mathcal{D}_{X_{P^{\mathrm{op}}}^\circ})\longrightarrow \mathcal{M}(W_{X_P}).
\end{align*}

\begin{Lem}\label{lem:localization}
Let $M$ be a nonzero left $W_{X_P}$-module. Then either $\mathcal{L}(M)$ or $\mathcal{L}_w(M)$ is nonzero.
\end{Lem}

\begin{proof}
     Suppose $\mathcal{L}(M)=0$. This implies $M$ is supported at the origin, i.e., for each $m\in M$ there is $r\ge 0$ such that one has $b(\mathbf{x})m=0$ for all $b\in \mathcal{B}_r$. Thus if both $\mathcal{L}(M)$ and $\mathcal{L}_w(M)$ are zero, then for each $m$ there is $r$ such that $b(\mathbf{x})m=b^{\mathrm{op}}(\partial)m=0$
    for all $b\in \mathcal{B}_r$. Corollary \ref{cor:critical} implies $M=0$.
\end{proof}

\begin{Thm}\label{thm:properties}
\begin{enumerate}
    \item The algebra $W_{X_P}$ is Noetherian.
    \item The injective dimension of $W_{X_P}$ as a left $W_{X_P}$-module is $\dim X_P$.
    \item The algebra $W_{X_P}$ is Auslander-Gorenstein and Cohen-Macaulay.
\end{enumerate}  
\end{Thm}

\begin{proof}
    Thanks to Lemma \ref{lem:localization}, (1) and (2) can be proved using the same argument as in \cite[Theorem 1.1 and Theorem 1.2]{BBP:Gluingbasicaffine}. For (3), one can proceed the proof as in \cite[Theorem 3.3]{Stafford} using the lemma below, which is a refinement of \cite[Lemma 3.2]{Stafford}.
\end{proof}

\begin{Lem}\label{lem:multiplicity}
    Let $M$ be a nonzero finitely generated left $W_{X_P}$-module. There is a positive integer $n$ and a constant $c_1>0$ such that for any good filtration $\Gamma=(\Gamma_k)$ of $M$ with respect to the Bernstein filtration
    \begin{align*}
         \dim \Gamma_k\le c_1k^{n}
    \end{align*}
    for $k$ large. Furthermore, there is an absolute constant $c>0$ independent of $M$ and $\Gamma$ such that
    \begin{align*}
        ck^{n}\le \dim \Gamma_k
    \end{align*}
    for $k$ large. In particular, $\mathrm{GK}(M)=n$ is an integer.
\end{Lem}

\begin{proof}
 Let $S_i:=\{ x_i^j: j\in \zz_{\ge 0}\}$ and  $U_i:=V(x_i)^c$. Let $F_i\subseteq S^{-1}_i W_{X_P}=\mathcal{D}(U_i)$ be a finite dimensional vector space spanned by a finite set of generators so that $\mathcal{F}_i=(F_i^k)_{k\ge 0}$ is a good standard filtration on $\mathcal{D}(U_i)$. Since $U_i$ is a smooth affine variety, we can choose $F_i$ such that $\mathrm{gr}^{\mathcal{F}_i}(\mathcal{D}(U_i))$ is a finitely generated commutative $K$-algebra (of dimension $2\dim X_P$).
 
 Let $S_i':=\{ (x_i^{\mathrm{op}})^j: j\in \zz_{\ge 0}\}.$ Choose analogously good standard filtrations $\mathcal{F}_i'=(F_i'^k)_{k\ge 0}$ on $S'^{-1}_i W_{X_{P^\mathrm{op}}}.$ By Lemma \ref{lem:localization}, we have an injection of $W_{X_P}$-modules
    \begin{align*}
        M\hookrightarrow \bigoplus_i S_i^{-1}M\oplus S_i'^{-1}M^w.
    \end{align*}
Choose $t>0$ such that
    \begin{align*}
        B_1\subseteq F_{i}^t, \quad \mathcal{F}_{P|P^{\mathrm{op}}}(B_1)\subseteq F_{i}'^t.
    \end{align*}
    Let $M_0$ be a vector subspace of $M$ spanned by a finite set of generators of $M$. Then we have
    \begin{align*}
        \dim B_1^kM_0\le \sum_{i} \dim F_{i}^{tk}M_0+\dim F_{i}^{'tk}M_0^w.
    \end{align*}
     Let $n:=\max_i (\mathrm{GK}(S_i^{-1}M), \mathrm{GK}(S_i'^{-1}M^w))$. Then $\dim B_1^kM_0\ll k^n$ by the theory of $D$-modules on smooth affine varieties.   This proves the first assertion upon observing that $c_1$ and $n$ do not depend on good filtrations on $M_0$.

     To prove the second assertion, we may assume $\mathrm{GK}(S_i^{-1}M)=n$ for some $i$. By definition, $x_i^rF_i\subset W_{X_P}$ for some $r\ge  0$. Choose a finite-dimensional vector subspace $V$ of $W_{X_P}$ containing $B_1$ such that $x_i^rF_i\subseteq V$ and $V$ is stable under the operator $\delta(\cdot):=[x_i,\cdot].$ Note that there is $q\in \zz_{>0}$ such that $\delta^q(V)=0.$ Arguing as in \cite[Lemma 15.4.2]{noncomm}, for any $k\in \zz_{\ge 0}$ we have
    \begin{align*}
        F_i^{k}\subseteq x_i^{-(r+q-1)k}V^{qk}.
    \end{align*} 
    Hence $\dim F_i^kM_0\le \dim V^{qk}M_0.$ Since $(F_i^kM_0)_{k\ge 0}$ is a standard filtration on $S^{-1}_iM$ with respect to $\mathcal{F}$, there is an absolute constant $c>0$ such that $ck^{n}\le\dim F_i^kM_0$ for all large $k$. Choose $s>0$ such that $V^{q}\subseteq B_1^{s}$. Then for $k$ large
    \begin{align*}
        ck^{n}\le \dim V^{qk}M_0\le \dim B_1^{sk}M_0.
    \end{align*}
    Since $s$ is independent of $M$, this completes the proof.
\end{proof}

Let $M$ be a nonzero finitely generated left $W_{X_P}$-module, so $\mathrm{GK}(M)=n\ge \dim X_P$ is a positive integer. Thanks to Lemma \ref{lem:multiplicity}, we can define the (unnormalized) multiplicity of $M$ to be
\begin{align*}
    e(M):=\liminf_{k\to \infty} \frac{\dim \Gamma_k }{k^n}\in \rr_{>0}.
\end{align*}
Note that $e(M)$ is uniformly bounded below. Also if $0\to M'\to M\to M''\to 0$ is an exact sequence. Then
\begin{align*}
    e(M)\ge e(M')+e(M'').
\end{align*}

Recall that a finitely generated $K$-algebra $R$ is left (resp. right) finitely partitive if for any finitely generated left (resp. right) $R$-module $M$, $\mathrm{GK}(M)$ is an integer, and there is an integer $n>0$ such that for any descending chain of $R$-submodule
\begin{align*}
    M=M_0\supset M_1\supset \ldots \supset M_m\supset M_{m+1}=0
\end{align*}
such that $\mathrm{GK}(M_i/M_{i+1})=\mathrm{GK}(M)$ for all $i$, one has $m\le n$. By the discussion above and Lemma \ref{lem:multiplicity}, one has

\begin{Prop}\label{prop:partitive}
    The Weyl algebra $W_{X_P}$ is left and right finitely partitive.\qed
\end{Prop}

\begin{Rem}
    Proposition \ref{prop:partitive} would follow immediately if there exists a good standard filtration on $W_{X_P}$ whose associated graded algebra is a commutative $K$-algebra of finite type. In addition, it would be interesting to know if this  filtration can be chosen so that the associated graded algebra is isomorphic to $K[T^\ast X_P^\circ].$  A problem relevant to the existence of such filtrations is the existence of an integer $r\ge 1$ such that for any $D_i\in B_1$
    \begin{align*}
            [[\ldots[[D_1,D_2],D_3],\ldots], D_{r+1}]\in B_r.
    \end{align*}
    It follows from \eqref{eq:realization} that $r\ge d_P$. We suspect that $r$ exists and one can take $r=d_P$. For $d_P=1,$ this follows from the product rule. For $d_P= 2,$ this can be easily verified (see Example \ref{ex} below). In this case one can take a good standard filtration $\mathcal{F}$ with
    \begin{align*}
        F_1=B_1+\mathfrak{g}+\mathfrak{m}_P^{\mathrm{ab}}.
    \end{align*}
    Over $\cc$, the resulting associated graded algebra is indeed isomorphic to $\cc[T^\ast X_P^\circ]$. 
\end{Rem}

Finally, as in \cite{BBP:Gluingbasicaffine} we describe $\mathcal{M}(W_{X_P})$ as a glued category $\mathcal{C}$ (denoted by $\mathcal{B}_\Phi$ in ibid.) An object in $\mathcal{C}$ is a pair $(\mathcal{E},\mathcal{E}_w)\in \mathcal{M}(\mathcal{D}_{X_P^\circ})\times \mathcal{M}(\mathcal{D}_{X_{P^{\mathrm{op}}}^\circ})$ together with two morphisms
\begin{align*}
    h&:\mathcal{E}\longrightarrow \mathcal{L}\circ \Gamma_w(\mathcal{E}_w),\quad\quad h_w:\mathcal{E}_w\to \mathcal{L}_w\circ\Gamma(\mathcal{E})
\end{align*}
such that the following two diagram commutes
\[
\begin{minipage}{0.45\textwidth}
\centering
\footnotesize
    \begin{tikzcd}
        \mathcal{E}\ar[r,equal]\ar[rd,"h",swap]& \mathcal{L}\circ \Gamma(\mathcal{E}) \ar[r]& \mathcal{L}\circ \Gamma_w\circ \mathcal{L}_w\circ \Gamma(\mathcal{E})\\
        & \mathcal{L}\circ \Gamma_w(\mathcal{E}_w)\ar[ru,swap,"(\mathcal{L}\circ \Gamma_w)(h_w)"]
    \end{tikzcd}
\end{minipage}
\hfill
\begin{minipage}{0.45\textwidth}
\centering
\footnotesize
    \begin{tikzcd}
        \mathcal{E}_w\ar[r,equal]\ar[rd,"h_w",swap]& \mathcal{L}_w\circ \Gamma_w(\mathcal{E}) \ar[r]& \mathcal{L}_w\circ \Gamma\circ \mathcal{L}\circ \Gamma_w(\mathcal{E})\\
        & \mathcal{L}_w\circ \Gamma(\mathcal{E})\ar[ru,swap,"(\mathcal{L}_w\circ \Gamma)(h)"]
    \end{tikzcd}
\end{minipage}
\]
Here the horizontal maps are the maps induced by adjunction morphisms. A morphism in $\mathcal{C}$ is a pair of morphism in $\mathcal{M}(\mathcal{D}_{X_P^\circ})\times \mathcal{M}(\mathcal{D}_{X_{P^{\mathrm{op}}}^\circ})$ that is compatible with morphisms $(h,h_w).$ In words, $\mathcal{C}$ consist of pairs of sheaves whose global sections have compatible localizations with respect to $w.$

\begin{Thm}\label{thm:equivalence}
    We have a natural equivalence of categories
    \begin{align*}
        \mathcal{M}(W_{X_P})&\longrightarrow\mathcal{C}\\
        M&\longmapsto (\mathcal{L}(M),\mathcal{L}_w(M)).
    \end{align*}
\end{Thm}

\begin{proof}
As explained in \cite[Theorem 3.5]{BBP:Gluingbasicaffine}, this follows from Lemma \ref{lem:localization} and \cite[Corollary 2.7]{BBP:Gluingbasicaffine}.
\end{proof}

\subsection{Holonomic modules and the Bernstein-Sato polynomial}

\begin{Def}
A finitely generated left $W_{X_P}$-module $M$ is \textbf{holonomic} if it is zero or $\mathrm{GK}(M)=\dim X_P$.
\end{Def}

\begin{Prop}\label{prop:basicholonomic}
A holonomic $W_{X_P}$-module is Noetherian, Artinian, torsion, and cyclic.  The category of holonomic modules on $X_P$ is abelian, Noetherian and Artinian.
\end{Prop}

\begin{proof}
   By a standard argument, this follows from Proposition \ref{prop:easy}(3), Theorem \ref{thm:properties}(1) and Proposition \ref{prop:partitive}.
\end{proof}

The following lemma follows from a minor modification of the argument in \cite[Lemma 10.3.1]{C:Dmod} and Lemma  \ref{lem:multiplicity}.
\begin{Lem}\label{lem:growth}
Let $M$ be a left $W_{X_P}$-module and $\Gamma$ be a filtration on $M$ with respect to the Bernstein filtration. Suppose there is $c>0$ such that
\begin{align*}
    \dim \Gamma_k\le ck^{\dim X_P}
\end{align*}
for all $k$ large. Then $M$ is holonomic. \qed
\end{Lem}



Let $\varphi\in K[X_P]$ be nonzero. Let $s$ be a new variable. Note that the action of $W_{X_P}(K(s))$ on $K(s)[X_P]$ naturally extends to an action on $K(s)[X_P][\varphi^{-1}]$.  Let $\varphi^s$ be a formal symbol. Let $K(s)[X_P][\varphi^{-1}]\varphi^s$ be the space whose underlying space is $K(s)[X_P][\varphi^{-1}]$. It is a left $K(s)[X_P]$-module by the usual function multiplication. Its module structure can be upgraded to a left $W_{X_P}(K(s))$-module by the usual rules of differentiation as follows.

 Since $V(x_i)^c$ is smooth, we can define a left $(W_{X_P}(K(s)))_{x_i}$-module structure on
\begin{align*}
    K(s)[X_P][\varphi^{-1},x_i^{-1}]\varphi^s
\end{align*}
by specifying that for $D\in (W_{X_P}(K(s)))_{x_i}$ of order $1$
\begin{align*}
      D\cdot \varphi^s:=s\frac{D(\varphi)}{\varphi} \cdot \varphi^s,
\end{align*}
and extending this action to higher order elements by the product rule. These modules glue to a $\mathcal{D}_{X_P^\circ}$-module $\mathcal{E}$ over $K(s)$. This equip $K(s)[X_P][\varphi^{-1}] \varphi^s$ with a left $W_{X_P}(K(s))$-module structure. By the same reasoning, we have a $W_{X_P}(K)$-linear automorphism $t$ on $K(s)[X_P][\varphi^{-1}]\varphi^s$ given by
\begin{align}\label{eq:t}
    t(s^i\cdot \varphi^s)=(s+1)^i\varphi\cdot \varphi^s
\end{align}
for any $i\in\zz_{\ge 0}.$


\begin{Lem}\label{lem:p-1holonomic}
Let $\varphi\in K[X_P]$ be nonzero. The left $W_{X_P}(K(s))$-module $K(s)[X_P][\varphi^{-1}]\varphi^s$ is holonomic.
\end{Lem}

\begin{proof}
Let $m=\deg(\varphi)$. For $k\in \zz_{\ge 0},$ set
\begin{align*}
    \Gamma_k:=\{f\varphi^{-d_Pk}\cdot \varphi^{s}: f\in K(s)[X_P] \textrm{ and }\deg f \le (md_P+1)k\}.
\end{align*}
Since $f\varphi^{-d_P k}=f\varphi^{d_P}\varphi^{-d_P(k+1)}$ and 
\begin{align*}
    \deg(f \varphi^{d_P})= \deg f+md_P \le (md_P+1) k +md_P < (md_P+1)(k+1),
\end{align*}
we have $\Gamma_k\subseteq \Gamma_{k+1}$ and $x_i\Gamma_k\subseteq \Gamma_{k+1}$ for all $i$. For differentiation, we have by \eqref{eq:realization} 
\begin{align*}
    \partial_i f\varphi^{-d_Pk}\cdot \varphi^s\in g\varphi^{-d_P(k+1)}\cdot\varphi^s.
\end{align*}
for some $g\in K(s)[X_P]$ such that
\begin{align*}
    \deg(g)\le \deg(f)+md_P -1 < (md_P+1)(k+1).
\end{align*}
Thus $\partial_i\Gamma_k\subseteq \Gamma_{k+1}$. Therefore, $\Gamma=(\Gamma_k)$ is a good filtration on $K(s)[X_P][\varphi^{-1}]$ with respect to the Bernstein filtration. Since for $k$ large
    \begin{align*}
        \dim \Gamma_k \le \sum_{j=0}^{(md_P+1)k}\dim V^\lor_{-j\omega_P}=H_{X_P}((md_P+1)k),
    \end{align*}
the assertion follows from Lemma \ref{lem:growth}.
\end{proof}


\begin{Thm}\label{thm:gcd}
  Let $\varphi\in K[X_P]$ be nonzero. There exists a nonzero $B(s)\in K[s]$ and a differential operator $D(s)\in W_{X_P}[s]$ such that
    \begin{align*}
        B(s)\cdot \varphi^s=D(s)\varphi\cdot \varphi^{s}.
    \end{align*}
\end{Thm}

\begin{proof}
    By Proposition \ref{prop:basicholonomic} and Lemma \ref{lem:p-1holonomic}, $W_{X_P}(K(s))\cdot \varphi^s\subseteq K(s)[X_P][\varphi^{-1}]\varphi^s$ is holonomic of finite length.  Therefore, the chain of descending modules $W_{X_P}(K(s))\varphi^k\cdot \varphi^s$ stabilizes for $k$ large. In particular, there is $k\ge 0$ such that
    \begin{align*}
        \varphi^k\cdot \varphi^s\in W_{X_P}(K(s))\varphi^{k+1}\cdot \varphi^s.
    \end{align*}
    Apply the isomorphism $t^{-k}$ \eqref{eq:t} and clearing out denominators, we see that there is a nonzero $B(s)\in K[s]$ such that
    \begin{align*}
        B(s)\cdot \varphi^s\in W_{X_P}(K[s])\varphi\cdot \varphi^s.
    \end{align*}
\end{proof}

Theorem \ref{thm:gcd} implies the set
\begin{align*}
    \left\{ B(s)\in K[s]: \textrm{there exists } D(s)\in W_{X_P}[s] \textrm{ such that } B(s)\cdot \varphi^s=D(s)\varphi\cdot \varphi^s\right\}
\end{align*}
is a nonzero ideal of $K[s].$ Let $b_\varphi(s)$ be the monic generator of the ideal. We refer to $b_\varphi(s)$ as the \textbf{Bernstein-Sato polynomial} of $\varphi$ on $X_P.$ 

\begin{Ex}\label{ex}

Let $n\ge 3.$ Let $X_n\subset \A^{2n}$ be the vanishing locus of $x_1x_2+\ldots +x_{2n-1}x_{2n}$. Then $X_n$ is the Braverman-Kazhdan space $X_P$ associated to $G=\mathrm{SO}(n,n)$ and $P=P_n$ is the stabilizer of an isotropic line. In this case $X_P$ is canonically isomorphic to $X_{P^{\mathrm{op}}},$ so we identify both with $X_n$. It is shown in \cite[Example 7.10]{Hsu:asymptotics} that the Schwartz space $\mathcal{S}(X_n(\rr))$ is the minimal representation contained in the degenerate principal series $\mathrm{Ind}_{P_{n+1}}^{\mathrm{SO}(n+1,n+1)}(1_{-1})$. Furthermore, the Weyl algebra $W_{X_n}(\cc)$ is induced from the action of the universal enveloping algebra $U((\mathfrak{so}_{2n+2})_\cc)$ on $\mathcal{S}(X_n(\rr))$. In other words, we have a natural surjection
\begin{align*}
    U((\mathfrak{so}_{2n+2})_\cc) \longrightarrow W_{X_n}(\cc).
\end{align*}
Equip $U((\mathfrak{so}_{2n+2})_\cc) $ with the PBW filtration. The induced filtration $\mathcal{F}$ on $W_{X_n}(\cc)$ is standard with 
\begin{align*}
    F_1=B_1+\mathfrak{so}_{2n}(\cc)+\cc E.
\end{align*}
By \cite{Josephideal} $\mathrm{gr}^{\mathcal{F}}(W_{X_n}(\cc))\cong \cc[T^\ast X_n^\mathrm{sm}]$  as explained in \cite[Remark 4.3]{FuLiu}.

Let us compute the Bernstein-Sato polynomial of $\varphi(\mathbf{x})=x_1$ on $X_n$. Observe that since $d_{P}=2,$ we have $[\partial_i,x_j]\in \cc+\mathfrak{so}_{2n}(\cc)+\cc E$ for all $i,j$ by \eqref{eq:realization}. As $[x_i,Y]\in \mathrm{span}_\cc\{x_j\}$ for $Y\in \mathfrak{so}_{2n}(\cc)+\cc E,$ dually $[\partial_i,Y]\in \mathrm{span}_\cc\{\partial_j\}$ for all $i$. Recall $\Delta(X_n)\subset W_{X_{n}}$ is the subalgebra of $W_{X_{n}}$ generated by $\cc[X_n]$ and $\mathfrak{so}_{2n}(\cc)+\cc E.$ Then each element in $W_{X_n}$ can be written as $ \sum_i D_ib_i(\mathbf{\partial})$ for some $b_i\in\cc [X_n]$ and $D_i\in \Delta(X_n)$. 

We have $\partial_1\cdot x_1^{s+1}=J(s+1)x_1^s.$ On the other hand, suppose 
\begin{align*}
    B(s)\cdot \varphi^s=\sum_{i,j} s^iD_{ij}b_{ij}(\partial)\varphi\cdot \varphi^{s}
\end{align*}
for some $B\in \cc[s]$ and $D_{ij}\in \Delta(X_n), b_{ij}\in \cc[X_n].$ Note that given $b\in \cc[X_n]$ there is $b'\in \cc[x_1]$ such that for any $n\ge 0,$ $b(\partial)(x_1^n)=b'(\partial_1)(x_1^n)$.  Thus we may assume $b_{ij}(\partial)=\partial_1^j.$ Since elements in $\Delta(X_n)$ are contained in the nonnegative weight spaces of $[E,\cdot],$ $D_{i0}=0$ for all $i$. It follows that $J(s+1)|B(s),$ and thus $b_\varphi(s)=J(s+1)$.
\end{Ex}

\subsection{Applications to zeta integrals}

For locally integrable functions $f_1,f_2$ on $X_P^\circ(\rr)$ define pairings
\begin{align*}
    \langle f_1,f_2\rangle&:=\int_{X_P^\circ(\rr)} f_1(x)f_2(x)dx,\\
    (f_1,f_2)&:=\langle f_1,\bar{f}_2\rangle,
\end{align*}
whenever the integral is absolutely convergent. Notice that for $f_1,f_2\in L^2(X_P(\rr))$
\begin{align*}
    (f_1,f_2)=(\mathcal{F}_{P|P^{\mathrm{op}}}(f_1),\mathcal{F}_{P|P^{\mathrm{op}}}(f_2)).
\end{align*}

 Let $f\in C^\infty_c(X_P^\circ(\rr))$ and $\varphi\in C^\infty(X_P^\circ(\rr)).$ Let $D\in W_{X_P}(\cc)$. We claim
\begin{align*}
     \langle D(\varphi), f\rangle=\langle \varphi, D^{\ast}(f) \rangle   
\end{align*}
for some $D^\ast\in W_{X_P}(\cc)$. This is clear if $D\in \cc[X_P],$ where $D^\ast =D$. To prove the claim, it suffices to show the existence of $\partial_i^*$ for each $i$. Choose $f'\in C^\infty_c(X_P^\circ(\rr))$ such that $f'|_{\mathrm{Supp}(f)}=1$. We have
\begin{align*}
     \langle \partial_i(\varphi), f\rangle&= \langle \partial_i(\varphi f'), f\rangle\\
     &=( \partial_i(\varphi f'), \bar{f})\\
     &=(  \mathcal{F}_{P|P^{\mathrm{op}}}(\partial_i(\varphi f')),  \mathcal{F}_{P|P^{\mathrm{op}}}(\bar{f}))\\
     &=(2\pi\sqrt{-1})^{d_P}( x_i^{\mathrm{op}}\mathcal{F}_{P|P^{\mathrm{op}}} (\varphi f'), \mathcal{F}_{P|P^{\mathrm{op}}}(\bar{f}))\\
     &=(2\pi\sqrt{-1})^{d_P}( \mathcal{F}_{P|P^{\mathrm{op}}} (\varphi f'), x_i^{\mathrm{op}}\mathcal{F}_{P|P^{\mathrm{op}}}(\bar{f}))\\
          &=(-1)^{d_P}( \mathcal{F}_{P|P^{\mathrm{op}}} (\varphi f'), \mathcal{F}_{P|P^{\mathrm{op}}}(\partial_i(\overline{ f})))\\
          &=(-1)^{d_P}( \mathcal{F}_{P|P^{\mathrm{op}}} (\varphi f'), \mathcal{F}_{P|P^{\mathrm{op}}}(\overline{\partial_i(f)}))\\
          &=(-1)^{d_P}( \varphi f', \overline{\partial_i(f)})\\
     &=\langle \varphi, (-1)^{d_P}\partial_i (f)\rangle.
\end{align*}
Hence $\partial_i^\ast=(-1)^{d_P}\partial_i.$ This is the integration by parts on $X_P(\rr)$.

\begin{Thm}\label{thm:Z+ext}
    Let $\varphi\in \rr[X_P]-\rr$ and $U:=\{\varphi>0\}.$ Let $f\in \mathcal{S}(X_P(\rr)).$ The integral 
    \begin{align*}
       Z_+(f,s):=\int_{U} f(x)\varphi(x)^s dx
    \end{align*}
    originally defined on $\mathrm{Re}(s)\ge 0$ extends to a meromorphic function on $\cc$. Furthermore, if we write $ b_\varphi(s)=\prod_{\lambda} (s+\lambda),$ and let
    \begin{align*}
       \gamma_\varphi(s):={\prod_{\lambda}\Gamma(s+\lambda)},
    \end{align*}
    where $\Gamma(s)$ is the gamma function. Then $\gamma_\varphi(s)^{-1}Z_+(f,s)$ is entire. 
\end{Thm}

\begin{proof}
 Assume first $f\in C^\infty_c(X_P^\circ(\rr)).$ Choose an ascending sequence of compact subsets $U_n$ of $U$ such that $U=\cup _n U_n$. Then $\mathbf{1}_{U_n} \to \mathbf{1}_U$ in $L^1(X_P(\rr))$. Let $b_\varphi$ be the Bernstein-Sato polynomial of $\varphi$ so that for all $x\in U$
 \begin{align*}
     b_\varphi(s) \cdot\varphi(x)^s=(D(s)\cdot\varphi^{s+1})(x)
 \end{align*}
 for some $D(s)\in W_{X_P}[s].$ Note that $\varphi(x)^s$ is well-define for all $x\in U,$ so this identity becomes an identity of holomorphic functions on $\mathrm{Re}(s)>0.$
 
 Then for $\mathrm{Re}(s)>0$
\begin{align*}
    b_{\varphi}(s)Z_+(f,s)&=\lim_{n\to \infty}\langle D(s)\varphi^{s+1} ,f\mathbf{1}_{U_n}\rangle \\
    &=\lim_{n\to \infty}\langle \varphi^{s+1} , D(s)^{\ast}(f\mathbf{1}_{U_n})\rangle \\
    &=\langle \varphi^{s+1}, D(s)^\ast (f)\mathbf{1}_U \rangle \\
    &=Z_+(D(s)^{\ast}f,s+1).
\end{align*}
For general $f\in \mathcal{S}(X_P(\rr))$, choose by smooth truncations $f_i\in C^\infty_c(X_P^\circ(\rr))$ converging to $f$ in $L^1(X_P(\rr)).$ Then by Lebesgue's dominated convergence theorem for $\mathrm{Re}(s)>0$
\begin{align*}
      b_\varphi(s)Z_+(f,s)=\lim_{i\to\infty} b(s)Z_+(f_i,s)=\lim_{i\to\infty} Z_+(D(s)^\ast f_i,s+1)=Z_+(D(s)^\ast f,s+1).
\end{align*}
Observe that RHS is absolutely convergent for $\mathrm{Re}(s)>-1.$ Therefore, by repeating the process  we can analytically extend the domain of definition of $Z_+(f,s)$ to all $\cc,$ with possible poles controlled by $b_\varphi(s+k)$ for $k\in \zz_{\ge 0}$ as claimed.
\end{proof}

Let $f\in \mathcal{S}(X_P(\rr))$ and $\varphi\in \rr[X_P]-\rr.$ For a quasi-character $\chi:\rr^\times\to \cc^\times,$ define the zeta integral
\begin{align*}
    Z(f,\chi):=\int_{X_P^\circ(\rr)} f(x)\chi(\varphi(x)) dx.
\end{align*}
Note that we can write $\chi=|\cdot|^{s}\mathrm{sgn}^j$ where $s\in \cc$ and $j\in \{0,1\}.$ For $\mathrm{Re}(s)>0$, write
\begin{align*}
     Z(f,\chi)=\int_{\varphi>0} f(x)\varphi(x)^s+(-1)^j\int_{\varphi<0} f(x)(-\varphi(x))^sdx.
\end{align*}
Since $b_\varphi=b_{-\varphi},$ by Theorem \ref{thm:Z+ext} $Z(f,\chi)$ can be extended to a meromorphic function in $\chi$ such that $\gamma_\varphi(s)^{-1}Z(f,|\cdot|^s\mathrm{sgn}^j)$ is entire. The same argument applies to $X_P(\cc)$ using the fact that $W_{X_P}(\cc\otimes_\rr\cc)=W_{X_P}(\cc)\oplus \overline{W_{X_P}(\cc)}$ under the complex conjugation (c.f. \cite[pp. 71-73]{Igusa:zeta}).

\bibliographystyle{alpha}

\bibliography{bibs}
\end{document}